\documentclass{amsart}
\usepackage{graphicx} 
\usepackage{tcolorbox}		
\usepackage{comment}	
\usepackage{ulem}
\usepackage{amsmath,amssymb,amsthm}
\usepackage{xcolor}
\usepackage[unicode=true,psdextra]{hyperref}
\newcommand{\sinc}{\,\mathrm{sinc\,}}
\theoremstyle{definition}
\newtheorem{definition}{Definition}
\newtheorem{remark}[definition]{Remark}
\theoremstyle{plain}
\newtheorem{theorem}[definition]{Theorem}
\newtheorem{lemma}[definition]{Lemma}

\newtheorem{prop}[definition]{Proposition}

\newcommand{\sign}{\text{sign}}

\newcommand{\dd}{\text{d}}
\newcommand{\kappab}{\kappa_0}
\newcommand{\added}[1]{#1}
\newcommand{\deleted}[1]{}

\title{Conditional Stability of the Euler Method on Riemannian Manifolds}
\author{Marta Ghirardelli}
\email{marta.ghirardelli@ntnu.no}
\author{Brynjulf Owren$^{*}$}
\thanks{$^{*}$Corresponding author}
\email{brynjulf.owren@ntnu.no}
\author{Elena Celledoni}
\email{elena.celledoni@ntnu.no}
\address{Department of Mathematical Sciences, NTNU, N-7491 Trondheim, Norway}

\date{December, 2025}



\begin{document}


\begin{abstract}
We derive nonlinear stability results for numerical integrators on Riemannian manifolds, by imposing conditions on the ODE vector field and the step size that makes the numerical solution non-expansive whenever the exact solution is non-expansive over the same time step.
Our model case is a geodesic version of the explicit Euler method. Precise bounds are obtained in the case of Riemannian manifolds of constant sectional curvature. The approach is based on a cocoercivity property of the vector field adapted to manifolds from Euclidean space. It allows us to compare the new results to the corresponding well-known results in flat spaces, and in general we find that a non-zero curvature will deteriorate the stability region of the geodesic Euler method. The step size bounds depend on the distance traveled over a step from the initial point. Numerical examples for spheres and hyperbolic 2-space confirm that the bounds are tight.

\end{abstract}
\maketitle

\section{Introduction}

In this paper, we propose a framework for analyzing the stability of numerical integrators on Riemannian manifolds and present results for a geodesic version of the explicit Euler method. This work is a continuation of the article \cite{arnold24bso}, in which the concept of B-stability for numerical methods in Euclidean spaces was extended to Riemannian manifolds.


\paragraph{\bf Remarks on the history of stability in Euclidean space.}
The stability of numerical methods applied to linear systems is the most basic notion of stability. When considering constant coefficient, diagonalizable matrices, one can argue, via a change of variables, that it suffices to apply the method to the scalar differential equation $y'=\lambda y$, where $\lambda$ is a complex parameter.
For most numerical methods, there exists an associated stability region $\Omega\subset\mathbb{C}$, ensuring stability whenever $h\lambda\in\Omega$, where $h$ is the time step used by the numerical method.
A method is called unconditionally stable or A-stable if $\Omega\supset\mathbb{C}^-$. 
This form of stability is widely known, for an overview see for instance \cite{dekker1984,hairer96sod2,iserles91os}. 

A more general form of stability analysis applicable to nonlinear problems was developed in the mid-1970s. This theory is related to {\em norm contractivity}, which states that a method is stable if it does not increase the distance between two initial values over a time step, provided the exact solutions do not expand in the same norm.
More precisely, for two initial values $y_0$ and $z_0$ in a normed linear space, one has for the ODE vector field $X$
$$
 \|\exp(hX)y_0 - \exp(hX)z_0\| \leq \|y_0-z_0\|\quad\Rightarrow\quad \|y_1-z_1\| \leq \|y_0-z_0\|.
$$
If the chosen norm is an inner product norm and this condition holds for all positive stepsizes, $h$, the (one-step) method is said to be B-stable \cite{butcher1975, burrage1979}.
For multi-step methods, the norm contractivity must be considered in an extended space and gives rise to the notion of G-stability for one-leg methods, see
\cite{dahlquist1975} and the earlier works \cite{liniger1956, dahlquist1963}.
A crucial aspect of this theory is the monotonicity of the vector field $X$:
\begin{equation} \label{mono_euc}
\langle X|_y - X|_z, y-z\rangle \leq -\nu \|y-z\|^2,\ \nu\in\mathbb{R},
\end{equation}
which implies that $$\|\exp(tX)y-\exp(tX)z\| \leq e^{-\nu t} \|y-z\|.$$ Thus, the flow of $X$ is guaranteed to be non-expansive if the monotonicity condition holds with $\nu\geq 0$.

The monotonicity condition \eqref{mono_euc} works well in the case that $X$ is allowed to be arbitrarily stiff, but is less useful for the non-stiff situation where explicit schemes may be preferable. Dahlquist and Jeltsch \cite{dahlquist1979gdo} proposed to replace the monotonicity condition by the {\it cocoercivity condition}
\begin{equation} \label{coco_euc}
    \langle X|_y - X|_z, y-z\rangle \leq -\alpha \|X|_y-X|_z\|^2.
\end{equation}
Clearly, when $\alpha=\nu=0$ the conditions \eqref{mono_euc} and \eqref{coco_euc} are the same,
but if \eqref{coco_euc} holds with  $\alpha>0$ then $X$ is Lipschitz with
$\text{Lip}(X) \leq \frac{1}{\alpha}$ and non-expansive.  Dahlquist and Jeltsch invented the notion of $r$-circle contractive Runge--Kutta methods and were able to give a bound on the stepsize in terms of $\alpha$ for which such a method is non-expansive.

\paragraph{\bf Numerical integrators on manifolds.}
There are many different ways of devising numerical integrators for manifolds. It is impossible to describe vector fields and integrators without assuming some structure, and method classes often differ depending on which type of structure is imposed. For instance, for differential algebraic equations, the manifold comes with an ambient space and a set of constraints. Distances are usually measured in the metric of the ambient space. Matrix Lie groups are modeled as subgroups of the general linear group. Crouch and Grossman \cite{crouch93nio} were among the first to propose numerical integrators that are intrinsically defined on manifolds. In their case, the added structure was a frame of vector fields in terms of which the ODE vector field could be expressed, and the methods were constructed by taking compositions of flows of constant linear combinations of the frame. Munthe-Kaas \cite{munthe-kaas99hor} considered homogeneous spaces where a central object is the transitive action of a Lie group on the manifold. Supplied with such an action together with the exponential map, one can represent the vector field locally on the Lie algebra of the group and also use the group action to evolve the numerical approximation on the manifold. A common framework that can be used to study these integrators geometrically and algebraically is through the use of a certain flat affine connection with non-zero torsion  \cite{munthe-kaas08oth, lundervold15oas}. For a practical introduction to Lie group integrators in general, see e.g. \cite{christiansen2011topics, celledoni14ait, owren2018lie}.

\paragraph{\bf Related work.} In \cite{simpson-porco14cto}, Simpson--Porco and Bullo provide several results in the contraction theory for Riemannian manifolds, with applications to dynamical systems which are of interest in optimization and control theory. There, they introduce what we refer to as the monotonicity condition on Riemannian manifolds. Bullo et al. discuss contractive processes on Riemannian manifolds in \cite{bullo21fct}. Of particular interest to our work is Theorem 6, adapted from \cite{wang10maa}, which gives sufficient conditions on the stepsize for the Euler method to converge to an equilibrium point for a vector field which is Lipschitz and satisfies a monotonicity condition. 

Except from \cite{arnold24bso} we have found little about stability of numerical methods for solving ordinary differential equations on Riemannian manifolds. But in a paper by Kunzinger et al. \cite{kunzinger06gge} some Gronwall type estimates were derived for flows of vector fields, and these results were later used in a numerical setting by
Curry and Schmeding \cite{curry20col} and Celledoni et al. \cite{celledoni20epm} to derive convergence results and global error estimates.
As for conditionally non-expansive integrators in Euclidean space, there have been some recent developments building on \cite{dahlquist1979gdo}.
Sanz-Serna and Zygalakis \cite{sanz-serna20cof} investigated the non-expansive behaviour of Runge--Kutta methods applied to convex gradient systems. They found an instance of an explicit Runge--Kutta method which is expansive for any fixed positive stepsize when applied to a carefully constructed gradient field of a convex potential. Also, Sherry et al. \cite{sherry24dsn} applied the theory in \cite{dahlquist1979gdo} to develop a ResNet type neural network architecture with guaranteed 1-Lipschitz behaviour, see also \cite{celledoni23dsb}.

\paragraph{\bf Our contributions.}
We believe that this is the first paper to address the conditional stability of numerical methods on Riemannian manifolds using only intrinsic measures such as the Riemannian distance function. Our approach rests on the adaptation of {\em cocoercivity} to Riemannian manifolds, inspired by the paper \cite{dahlquist1979gdo} where conditional stability of Runge--Kutta methods in Euclidean spaces is analysed based on a cocoercivity condition. A geodesic version of the explicit Euler method is analysed for the case that the manifold has constant sectional curvature, and we believe that the obtained bounds on the stepsize for the Euler map to be non-expansive are new.

In Section~\ref{sec:background} we introduce the tools and notation we are using, and provide some preliminary results which are needed later.
The stepsize bounds for conditional stability of the Euler method in the constant curvature case are presented and proved in Section~\ref{sec:conditional}. The main results are summarized in Theorem~\ref{theo: GEE non-expansive for positive rho} for manifolds of constant positive sectional curvatures, and Theorem~\ref{theo: GEE non-expansive for negative rho} gives the corresponding result for negatively curved (hyperbolic) spaces.
In Section~\ref{sec:examples} we present three concrete toy problems on Riemannian manifolds, one for each of the positively curved $S^2$ and $S^3$, and one
on the hyperbolic half plane $\mathbb{H}^2$, a case of constant negative sectional curvature.

\section{Background and preliminaries} 
\label{sec:background}

The notation and terminology we use is mostly adopted from Lee \cite{lee18itr}, and we refer the reader to this and other text books such as \cite{docarmo92rg,jost17rga} for details. A Riemannian manifold $(\mathcal{M}, g)$ is a smooth manifold $\mathcal{M}$ equipped with a smoothly varying inner product $g$ on each tangent space $T_p \mathcal{M}$, $p\in\mathcal{M}$. We write $g(\cdot, \cdot)=\langle \cdot, \cdot \rangle$, and $\|\cdot\|=\sqrt{g(\cdot,\cdot)}$ whenever convenient. \added{We also assume $\mathcal{M}$ is connected throughout the paper.}
The length of a curve $\gamma:[a,b]\subseteq\mathbb{R} \to \mathcal{M}$ is denoted by
\begin{equation*}
    L(\gamma) = \int_a^b \| \Dot{\gamma}(t) \| \, \text{d}t,
\end{equation*}
and for the Riemannian (geodesic) distance between points $p$ and $q$ on $\mathcal{M}$ induced by the metric, we write $d(p,q)$.
%
A subset $\mathcal{U}\subseteq\mathcal{M}$ is geodesically convex if, for any pair of points $p,q\in\mathcal{U}$, there exists a unique minimizing geodesic segment from $p$ to $q$ contained in $\mathcal{U}$.  We refer to $\mathfrak{X}(\mathcal{M})$ as the set of all smooth vector fields on $\mathcal{M}$. We reserve the symbol $\nabla$ for the Levi-Civita connection associated with $(\mathcal{M}, g)$, and write $\nabla_X Y$ for the covariant derivative of $Y$ along $X$. The connection defines the covariant derivative of a vector field $V\in\mathfrak{X}(\mathcal{\gamma})$ along a curve $\gamma(t)$, namely $D_t V(t)=\nabla_{\dot{\gamma}(t)}\tilde{V}$, where $\tilde{V}$ is an extension of $V$ to a neighborhood of $\gamma(t)$.  
Whenever $\gamma(t)$, $t\in [0,t^*]$ for some $t^*>0$, is the solution of the second order differential equation $D_t \Dot{\gamma}(t)=0$, it is called a geodesic. When initial conditions $\gamma(0)=p\in\mathcal{M}$ and $\Dot{\gamma}(0)=v_p \in T_p\mathcal{M}$ are imposed, we obtain a unique local solution $\gamma(t)$, 
which defines the map $\exp_p : T_p \mathcal{M}\to \mathcal{M}$ as $v_p\mapsto \exp_p(v_p)=\gamma(1)$. We similarly write $\exp(tX)$ (without the subscript on '$\exp$') for the $t$-flow of a vector field $X\in\mathfrak{X}(\mathcal{M})$, that is the diffeomorphism $p \mapsto y(t) = \exp(tX) \, p$ with $y(0)=p\in\mathcal{M}$, $\Dot{y}=X|_y$. 

A natural way of defining numerical integrators on Riemannian manifolds is through the use of geodesics, compositions of geodesics and parallel transport. We refer to such methods as geodesic integrators. They are all designed to approximate the flow map $\exp(tX)$.
The simplest of them is a geodesic version of the explicit Euler method, which we refer to as the Geodesic Explicit Euler (GEE) method. It is given by
\begin{equation}
    \label{eq: GEE}
    \tag{GEE}
    y_{n+1} = \exp_{y_n}(h X|_{y_{n}}).
\end{equation}
A similar definition for the Geodesic Implicit Euler (GIE) method reads
\begin{equation}
    \label{eq: GIE}
    \tag{GIE}
    y_{n} = \exp_{y_{n+1}}(-h X|_{y_{n+1}}).
\end{equation}
Both of the schemes (\ref{eq: GEE}) and (\ref{eq: GIE}) have convergence order one, and they reduce to the classical explicit and implicit Euler when the manifold is the Euclidean space \cite{arnold24bso}.
As far as we know, there is no unique natural way of generalizing the GEE and GIE methods to obtain higher order schemes, for instance in the spirit of Runge-Kutta methods. Clearly, this is possible, e.g. by constructing a retraction map based on normal coordinates \cite{celledoni02aco}, but the resulting methods can be somewhat complicated. Simple, naive method formats may, on the other hand, lead to order barriers.
Alternative numerical integrators on manifolds are available in the literature. Crouch and Grossman built integrators by composing the flows of ``frozen" vector fields \cite{crouch93nio}. These were further developed by Celledoni et al. into commutator-free schemes \cite{celledoni03cfl, curry19vss}. Munthe-Kaas showed that any classical Runge-Kutta method can be turned into a method of the same order on a general homogeneous manifold \cite{munthe-kaas98rkm}. The resulting integrators are known as Runge-Kutta Munthe-Kaas (RKMK) methods. Many variants of these methods can be found in the literature, see for instance the surveys
\cite{iserles2000, celledoni14ait, christiansen2011topics, owren2018lie, celledoni21lgi}.

A map $\varphi:\mathcal{M}\to\mathcal{M}$ is non-expansive on $\mathcal{U} \subseteq \mathcal{M}$ 
if
\begin{equation*}
    d\big( \varphi(x_0), \varphi(y_0) \big) \le d(x_0, y_0)
\end{equation*}
for any $x_0, y_0 \in \mathcal{U}$. Choosing $\varphi$ to be the $t$-flow of a vector field, or its numerical approximation, one can get non-expansive flows and non-expansive numerical schemes. We denote by $\nabla X$ the linear operator ((1,1) tensor field) acting fiberwise on $T\mathcal{M}$; $\left.\nabla X\right|_p : v_p \mapsto \nabla_{v_p} X $. The flow of a vector field $X$ is non-expansive
if $X$ satisfies the monotonicity condition
\begin{equation}
    \label{eq:monotonicity-condition}
    \langle \nabla_{v_p} X, v_p \rangle \le -\nu \, \| v_p  \|^2,\ \forall \, p\in\mathcal{U}\ \text{and}\ v_p\in T_p \mathcal{M}
\end{equation}
for some $\nu\geq 0$ \cite{arnold24bso, simpson-porco14cto}. The constant $\nu$ can then be chosen from 
\begin{equation*}
    -\nu = \sup_{p\in\mathcal{U}} \mu_g (\nabla X|_p),
\end{equation*}
where $\mu_g$ is the logarithmic $g$-norm of $\nabla X|_p$, which for a linear operator $A:T_p\mathcal{M}\to T_p\mathcal{M}$ can be defined as \cite{dekker1984, söderlind2024logarithmic}
\begin{equation}
    \label{eq: logarithmic g-norm}
    \mu_g(A) = \sup_{0\ne v\in T_p\mathcal{M}} \frac{g(Av, v)}{g(v,v)}.
\end{equation}
Inspired by \cite{dahlquist1979gdo},
we shall replace \eqref{eq:monotonicity-condition} by a {\em cocoercivity condition}:
\begin{definition}\label{def:cocoercivity}
A linear operator $A: V\rightarrow V$ on the Hilbert space $V$ is {\em $\alpha$-cocoercive}
if there exists $\alpha>0$ 
such that
\begin{equation}
    \label{eq:cocoercivity-condition}
    \langle A v, v \rangle \le -\alpha \, \| Av \|^2\quad \forall \, v\in V.
\end{equation}
A $(1,1)$-tensor field $\mathcal{A}$ defined on $\mathcal{U}\subset \mathcal{M}$, where $\mathcal{M}$ is a Riemannian manifold, is called $\alpha$-cocoercive on $\mathcal{U}$ if $\mathcal{A}|_p : T_p\mathcal{M}\rightarrow T_p\mathcal{M}$ is $\alpha$-cocoercive for every $p\in\mathcal{U}$.
For convenience, we say that the vector field $X$ is $\alpha$-cocoercive on $\mathcal{U}$ whenever $\nabla X$ is $\alpha$-cocoercive on $\mathcal{U}$.
\end{definition}
See \cite[Chapter~4.2]{bauschke2017correction} for an in-depth treatment of cocoercive operators.
\footnote{Definition~\ref{def:cocoercivity} is adapted to our use from \cite{bauschke2017correction}, where the operator $A$ is allowed to be nonlinear, and
we apply their definition to the operator $-A$. Sometimes we also refer to \eqref{eq:cocoercivity-condition} for the case when it only holds with $\alpha\leq 0$.
}
\begin{remark}
    \label{remark: norm of nabla X}
    Let $A=\left.\nabla X\right|_p$ in Definition~\ref{def:cocoercivity}. Note that \eqref{eq:monotonicity-condition} with $\nu=0$ is equivalent to \eqref{eq:cocoercivity-condition} with $\alpha=0$.
    If (\ref{eq:cocoercivity-condition}) holds with $\alpha\geq 0$, then \eqref{eq:monotonicity-condition} is satisfied with $\nu=0$.
    Also, when $X$ is $\alpha$-cocoercive on $\mathcal{U}$ (i.e. $\alpha>0$) we obtain by 
    the Cauchy-Schwarz inequality that
    \begin{equation*}
        \alpha \| \nabla_{v_p} X \|^2 \le - \langle \nabla_{v_p} X , v_p \rangle \le \| \nabla_{v_p} X \| \, \| v_p \|,
    \end{equation*}
    leading to the bounds $\| \nabla_{v_p} X \| \le \frac{1}{\alpha} \| v_p \|$, and $\|\nabla X \| \le \frac{1}{\alpha}$ where $\| \cdot \|$ is the operator norm.
\end{remark}

\begin{prop}
	\label{prop: 2nd condition implies nonexpansivity}
    Let $(\mathcal{M}, g)$ be a \added{connected} Riemannian manifold, $\mathcal{U}\subset\mathcal{M}$ an open, geodesically convex set, and let $X\in\mathfrak{X}(\mathcal{M})$ be $\alpha$-cocoercive on $\mathcal{U}$ for some $\alpha > 0$. Then, for any $x_0, y_0\in\mathcal{U}$ there exists $t^*>0$ such that
    $$d(\exp(tX) x_0, \exp(tX) y_0) \leq d(x_0, y_0)$$ for $0<t\le t^*$.
\end{prop}
\begin{proof}
From  \cite[Theorem~2.2]{arnold24bso}, we see that if \eqref{eq:monotonicity-condition} holds for some $\nu\in\mathbb{R}$, then 
$$
d(\exp(tX) x_0, \exp(tX) y_0) \leq e^{-\nu t} d(x_0, y_0),\ x_0, y_0\in\mathcal{U},\ 0<t\le t^*,
$$
for some $t^*>0$. The result follows because if (\ref{eq:cocoercivity-condition}) holds for $A=\left.\nabla X\right|_{p}$ with $\alpha > 0$, then \eqref{eq:monotonicity-condition} is satisfied with $\nu=0$ (Remark~\ref{remark: norm of nabla X}). 
\end{proof}
\bigskip


We now prove a lemma that provides a sufficient condition for the GEE method to be non-expansive.
\added{
\begin{lemma} \label{lemma4}
    Consider $x_0,y_0 \in \mathcal{M}$ and let $y: [s_1,s_2]\rightarrow \mathcal{M}$ be any admissible curve such that $y(s_1)=x_0,\ y(s_2)=y_0$. Suppose $\phi:\mathcal{M} \to \mathcal{M}$ is a smooth map, and denote $z(s) = \phi(y(s))$. If the tangents $z'(s)$ and $y'(s)$ satisfy
    \begin{equation}
        \label{nrmdiff00}
        \| z'(s) \| - \| y'(s) \| \le 0, \quad \forall \, s \in [s_1,s_2]
    \end{equation}
    then
    $$
    L(z(\cdot{})) \leq L(y(\cdot)).
    $$ 
    Moreover, if $y(s)$ is a length-minimizing geodesic, then
    $$
    d(x_1,y_1) \leq d(x_0,y_0),
    $$
    where $x_1=\phi(x_0)$, $y_1 =  \phi(y_0)$.
\end{lemma}
\begin{proof}
    The first part of the Lemma follows from the definition of the length of a curve. We have
    $$
    L(z(\cdot{})) - L(y(\cdot)) = \int_{s_1}^{s_2} \| z'(s) \| \dd s - \int_{s_1}^{s_2} \| y'(s) \| \dd s = \int_{s_1}^{s_2} \|  z'(s)\| - \| y'(s) \|\, \dd s \le 0.
    $$
    Assume now that $y(s)$ is also a length-minimizing geodesic for which \eqref{nrmdiff00} holds. Then
    $$
    d(x_1,y_1) \leq L(z(\cdot{})) \le L(y(\cdot)) =  d(x_0,y_0).
    $$
\end{proof}
Let $y: [s_1,s_2]\rightarrow \mathcal{M}$ be a length-minimizing geodesic connecting two points $x_0$ and $y_0$. We assume that we can define a smooth variation through geodesics
\begin{equation}
    \label{eq: variation}
    \Gamma(s,t)=\exp_{y(s)}(t h X|_{y(s)}),\ s\in [s_1,s_2],\ t\in[0,1].
\end{equation}
Then $\Gamma(s,1)=\exp_{y(s)}(hX|_{y(s)})$ is precisely a step with GEE with initial values $y(s)$. With $S^s(t) = \partial_s \Gamma(s,t)$, condition \eqref{nrmdiff00} takes the form
\begin{equation} \label{nrmdiff0}
    \|S^s(1)\|-\|S^s(0)\|\leq 0,\ \forall \, s\in [s_1,s_2].
\end{equation}
}

The vector field $S^s(t)$ \deleted{appearing in Lemma~\ref{lemma4}} is a  {\it Jacobi field} along the \added{geodesic $t\mapsto \Gamma(s,t)$ for every fixed $s$. 
This is a crucial observation in what follows.}
It is well-known that such Jacobi fields
satisfy a linear second order differential equation called the Jacobi equation, it reads
\begin{equation} \label{jacobiequation}
	D_t^2 S^s + R(S^s, T^s) T^s = 0,
\end{equation}
where $T^s(t)=\partial_t\Gamma(t,s)$ and $R$ is the (Riemann) curvature tensor, $$R(X,Y)Z=\nabla_X\nabla_YZ-\nabla_Y\nabla_X Z -\nabla_{[X,Y]}Z.$$
\added{The curvature tensor $R$ vanishes identically when the manifold is Euclidean space, and the general solution to the Jacobi equation is
$S^s(t)=S^s(0) + t \dot{S}^s(0)$}.
The Jacobi equation must be supplied with initial values $S^s(0)$ and $D_tS^s(0)$ to determine a unique solution. Given $\Gamma$ as in \eqref{eq: variation},
\begin{align}
    \label{eq: IC from Gamma}
    \begin{split}
        S^s(0) &= (\partial_s y) (0) = S_0, \\
        D_t S^s (0) &= (D_s T^s) (0) = h \, D_s X |_{s=0} = h \nabla_{S_0} X,
    \end{split}
\end{align}
where we have used the symmetry of the Levi-Civita connection $\nabla$. 

The sectional curvature  of a two-dimensional subspace of the tangent space $T_p \mathcal{M}$ at $p\in\mathcal{M}$ is defined as 
$$K(u,v)=\frac{\langle R(u,v)v,u \rangle}{\langle u,u\rangle \langle v,v\rangle- \langle u,v\rangle^2},\quad u,v\in T_p\mathcal{M},$$
for $u,v$ linearly independent. It depends only on $\text{span}\{u,v\}$ and the point $p\in\mathcal{M}$.
If $K$ is equal to a constant $\rho$ 
 for all two dimensional subspaces $\text{span}\{u,v\}\subset T_p\mathcal{M}$ $\forall p\in \mathcal{M}$, then $\mathcal{M}$ has constant sectional curvature $\rho$. Examples  
are the Euclidean space $(\rho=0)$, the $n$-sphere of radius $r$ $(\rho=1/r^2)$, the $n$-hyperbolic space of radius $r$ $(\rho=-1/r^2)$, and any 
Riemannian manifold isometric to one of them.  

In Section~\ref{sec:conditional}, $\mathcal{M}$ will assume constant sectional curvature $\rho$ and the curvature tensor $R$ in this case takes the form
\begin{equation*}
    \label{eq: R for constant rho}
    R(v, w) x = \rho \, (\langle w, x \rangle v - \langle v, x \rangle w)
\end{equation*}
for any $v, w, x \in T_p \mathcal{M}$, $p \in \mathcal{M}$ \cite[Proposition~8.36]{lee18itr}.

\section{Conditional stability of the Geodesic Explicit Euler method in the constant curvature case.}
\label{sec:conditional}

The Jacobi equation \eqref{jacobiequation} has closed form solutions in the case that the manifold $\mathcal{M}$ has constant sectional curvature, for details, see \cite[Chapter~6.5]{jost17rga}. 
For the geodesic $\gamma(t)$, let $\{e_1(t),\ldots,e_d(t)\}$ be a parallel orthonormal frame along $\gamma(t)$ where $\dot{\gamma}(t)=\|\dot{\gamma}(t)\| e_1(t)$,
and suppose that $\rho$ is the sectional curvature of $\mathcal{M}$.
Define $\kappa=\sqrt{|\rho|} \|\dot{\gamma}(0)\|$.

Then any Jacobi field $J(t)$ is of the form
\begin{equation}
    \label{eq: Jacobi fields formula}
    J(t) = (a_1 + b_1 t) \, e_1(t) + \sum_{i=2}^d \left(a_i c_{\kappa }(t) + b_i s_{\kappa}(t) \right) e_i(t),
\end{equation}
where
\begin{equation}
\label{CkSk}
    c_{\kappa}(t) = \begin{cases}
        \cos(\kappa t) & \text{if } \rho >0, \\
        1 & \text{if } \rho =0, \\
        \cosh(\kappa t) & \text{if } \rho <0,
    \end{cases} \qquad
     s_{\kappa}(t) = \begin{cases}
        \frac{\sin(\kappa t)}{\kappa} & \text{if } \rho >0, \\
        t & \text{if } \rho =0, \\
        \frac{\sinh(\kappa t)}{\kappa} & \text{if } \rho <0,
    \end{cases}
\end{equation}
see e.g. \cite[Chapter~6.5]{jost17rga}. The coefficients $a_i \in \mathbb{R}$, $b_i \in \mathbb{R}$, $i=1,\dots,d$ are uniquely determined once the initial conditions $J(0)$ and $D_t J(0)$ are given \cite[Proposition~10.2]{lee18itr}. Differentiating (\ref{eq: Jacobi fields formula}) with respect to $t$ gives
\begin{equation*}
    D_t J(t) = b_1 \, e_1(t) + \sum_{i=2}^d \left(a_i \dot{c}_{\kappa }(t) + b_i \dot{s}_{\kappa}(t) \right) e_i(t),
\end{equation*}
and since $c_{\kappa}(0) = \dot{s}_{\kappa}(0) = 1$ and $\dot{c}_{\kappa}(0) = s_{\kappa}(0) = 0$, we obtain
\begin{equation}
    \label{eq: IC from formula}
    J(0) =   \sum_{i=1}^d a_i \, e_i(0), \qquad D_t J(0) =   \sum_{i=1}^d b_i \, e_i(0).
\end{equation}
Therefore, we also have
\begin{equation}
    \label{eq: coefficients expression}
    a_i = \langle J(0), e_i(0) \rangle, \qquad b_i = \langle D_tJ(0), e_i(0) \rangle. 
\end{equation}

In what follows, we will make use of the following non-negative functions:
\begin{eqnarray}
\label{f1}
f_1(\kappa)&:=&\sign(\rho)(1- c_{\kappa}^2(1)),\\ \label{f2}f_2(\kappa)&:=&\sign(\rho)( 1-c_{\kappa}(1)s_{\kappa}(1)), \\ \label{f3}f_3(\kappa)&:=&\sign(\rho) ( 1-s_{\kappa}^2(1)).
\end{eqnarray}
These functions are identically zero in the case of zero curvature.

The following lemma is useful for analyzing condition \added{\eqref{nrmdiff0}} \deleted{of Lemma~\ref{lemma4}} in the constant curvature case.
\begin{lemma}
\label{lemma1}
Let $(\mathcal{M},g)$ be a \added{complete, connected} $d$-Riemannian manifold with constant sectional curvature $\rho$, and let $p$ be a point in $\mathcal{M}$.
Suppose $u_p\in T_p\mathcal{M}$ and  $\gamma(t)=\exp_p(t u_p)$ with $0\leq t\leq 1$.
For the Jacobi field $J(t)$ along $\gamma(t)$ satisfying $J(0)=v_p$, $D_tJ(0)=w_p$ we have
\begin{align}
 \label{nrmdiff}
 \begin{split}
   \|J(1)\|^2 - \|J(0)\|^2 &=  \| w_p \|^2 + 2 \langle v_p, w_p \rangle \\ 
   &- \sign({\rho})\sum_{i=2}^d \big( a_i^2 f_1(\kappa) +2 a_i b_i f_2(\kappa) + b^2_i f_3(\kappa) \big),   
\end{split}
\end{align} 
where $a_i=\langle v_p, e_i\rangle$, $b_i=\langle w_p, e_i\rangle$, $\{e_1,\ldots,e_d\}$ is any orthonormal frame at $p$ such that
$u_p=\|u_p\| e_1$,  $f_i(\kappa)$ are given in \eqref{f1}--\eqref{f3} and \eqref{CkSk}, and $\kappa=\sqrt{|\rho|}\|u_p\|$.
\end{lemma}

\begin{proof}
A straightforward calculation, using the formula \eqref{eq: Jacobi fields formula} for the Jacobi field and the stated initial conditions yields
\begin{align*}
    \|&J(1)\|^2 - \|J(0)\|^2 = (a_1 + b_1)^2 + \sum_{i=2}^d (a_i c_{\kappa}(1)+b_i s_{\kappa}(1))^2 - \sum_{i=1}^d a_i^2 \\
    &= \sum_{i=1}^d b_i^2 + 2 a_i b_i + \sum_{i=2}^d \big( a_i^2 (c^2_{\kappa}(1) - 1)  +2 a_i b_i  (c_{\kappa}(1) s_{\kappa}(1)-1) + b^2_i  (s^2_{\kappa}(1)-1) \big) \\   
    &= \| w_p \|^2 + 2  \langle v_p, w_p \rangle - \sign({\rho})\sum_{i=2}^d \big( a_i^2 f_1(\kappa) +2 a_i b_i f_2(\kappa) + b^2_i f_3(\kappa) \big),
\end{align*}
where we have also used the definitions \eqref{f1}--\eqref{f3}.
\end{proof}
\added{We note that $(a_1+b_1t)e_1(t)$ in \eqref{eq: Jacobi fields formula} are  solutions to the Jacobi equation for any manifold. 
These are tangent to the geodesic $\gamma(t)$. The solutions in the orthogonal complement to $e_1(t)$ depend on the manifold and in particular on its curvature.
We shall find it useful to split any
tangent vector $v_p$ as $v_p=P_{X}v_p + (I-P_X)v_p$ where $P_X:T_{p} \mathcal{M} \to \text{span}(X|_{p})$  is the orthogonal projector for any $p\in\mathcal{M}$.
}

An important task is to define a suitable subset of the manifold where the conditions on the vector field is to be imposed.
For simplicity, we have here decided to impose the bounds over a geodesically convex subset, $\mathcal{U}\subset\mathcal{M}$. 
This condition can be slightly relaxed.

We are now ready to state the main result for conditional stability in the case of constant positive sectional curvature.

\begin{theorem}
\label{theo: GEE non-expansive for positive rho}
    Let $(\mathcal{M},g)$ be a \added{complete, connected} $d$-Riemannian manifold with constant sectional curvature $\rho>0$ and $\mathcal{U}\subset\mathcal{M}$ a geodesically convex subset.   
    Let $X\in\mathfrak{X}(\mathcal{U)}$ be an $\alpha$-cocoercive vector field on $\mathcal{U}$ for some $\alpha > 0$, and
    $\|X\|\leq C$ when restricted to $\mathcal{U}$.
    Suppose there exists a constant $\mu_+$ such that
    \begin{equation}
        \label{eq: 3rd monotonicty conditon positive rho}
        \langle \nabla_{v_p} X, (I - P_X) v_p \rangle \ge - \mu_+ \| \nabla_{v_p} X \|^2, \quad \forall \, p\in\mathcal{U}, \quad \forall \, v_p \in T_p \mathcal{M}.
    \end{equation}
 Then, for any choice of initial points $x_0, y_0\in\mathcal{U}$ the Geodesic Euler Method 
 satisfies
    \begin{equation} \label{toprove_pos_curvature}
        d \big( \exp_{x_0}(hX|_{x_0}), \exp_{y_0}(hX|_{y_0}) \big) \le d(x_0, y_0)
    \end{equation}
    whenever the step size $h$ is chosen such that
    \begin{equation}
    \label{eq: condition on h for positive rho}
   0 < h \le 2\alpha - 2\mu_+ \left( f_2(\kappab) - \sqrt{f_1(\kappab) f_3(\kappab)} \right),
    \end{equation}
    where $\kappab = h C \sqrt{\rho}$, and $f_1$, $f_2$, $f_3$ as in \eqref{f1}-\eqref{f3}.
\end{theorem}
\begin{figure}
    \centering
    \includegraphics[width=0.95\linewidth]{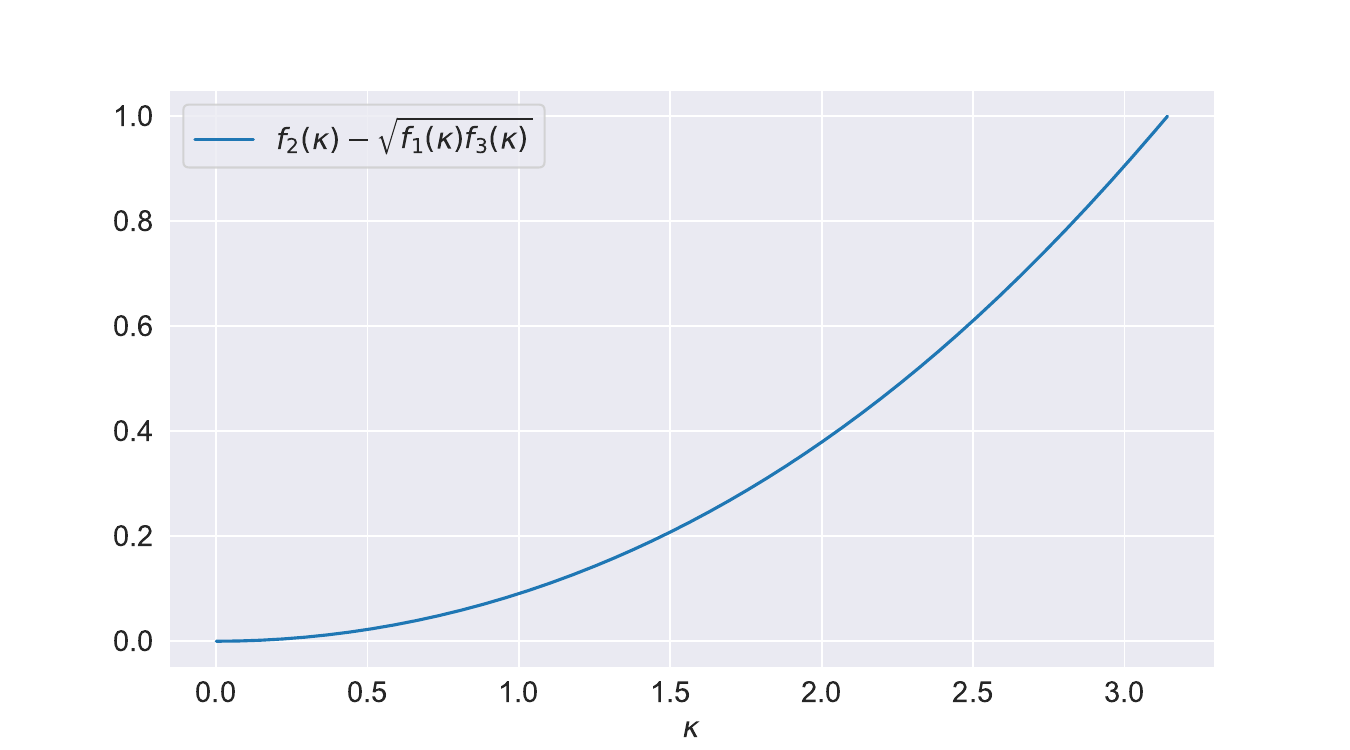}
    \caption{\small \deleted{Visualization} \added{The graph} of the function appearing in Theorem \ref{theo: GEE non-expansive for positive rho}, $f_2(\kappa) - \sqrt{f_1(\kappa) f_3(\kappa)} = (1-\cos(\kappa)\sinc(\kappa)) - | \sin(\kappa) | \sqrt{1-\sinc^2(\kappa)}$.}
    \label{fig: f_plot}
\end{figure}
\begin{proof}
Take any $x_0, y_0\in\mathcal{U}$ and let $y(s),\, s\in[s_1,s_2]$ be the geodesic satisfying $y(s_1)=x_0$ and $y(s_2)=y_0$.
Since $\mathcal{M}$ is \deleted{compact and therefore} geodesically complete, the variation $\Gamma(s,t)=\exp_{y(s)}(thX|_{y(s)})$ is well defined and 
we set $S^s(t)=\partial_s \Gamma(s,t)\in\mathfrak{X}(\Gamma)$. 
\added{We must show that \eqref{toprove_pos_curvature} is satisfied under the stated assumptions.
For this, we will use the condition \eqref{nrmdiff0} resulting from Lemma~\ref{lemma4} and then specialize to the constant curvature case by applying Lemma~\ref{lemma1}. Next we invoke the $\alpha$-cocoercivity condition and \eqref{eq: 3rd monotonicty conditon positive rho}.
This yields an inequality which depends on the stepsize $h$, and from there we derive the bound \eqref{eq: condition on h for positive rho} on $h$ for the inequality to hold.
}
By Lemma~\ref{lemma4} it is sufficient to ensure that the condition
$\|S^s(1)\|-\|S^s(0)\|\leq 0$ is satisfied for any $s\in[s_1,s_2]$ \added{as seen from \eqref{nrmdiff0}}. We apply Lemma~\ref{lemma1} 
\added{with $J(t)=S^s(t)$, so that $J(0)=S^s(0)=:S_0$ and $D_tJ(0)=D_tS^s(0)= h\,\nabla_{S_0}X$.}
Thus, we need to check that
\begin{equation} \label{p_thm6_inequality_to_check}
 h^2 \| \nabla_{S_0} X \|^2 + 2 h \langle S_0, \nabla_{S_0} X \rangle - \sum_{i=2}^d \big( a_i^2 f_1(\kappa) +2 a_i b_i f_2(\kappa) + b^2_i f_3(\kappa) \big)
\leq 0,
\end{equation}
where $f_1$, $f_2$ and $f_3$ are as in \eqref{f1}-\eqref{f3}, $\kappa=h\sqrt{\rho}\|X|_{y(s)}\|$ and $a_i, b_i$ are given in \eqref{eq: coefficients expression}.
\added{For the first two terms in \eqref{p_thm6_inequality_to_check},} we use the $\alpha$-cocoercivity of $X$, to obtain the bound 
\begin{equation*}
     h^2 \| \nabla_{S_0} X \|^2 + 2 h \langle S_0, \nabla_{S_0} X \rangle 
    \le
    h(h-2\alpha) \, \| \nabla_{S_0} X \|^2.
\end{equation*}
\added{For each term in the remaining sum in \eqref{p_thm6_inequality_to_check}, we use a variant of
Young's inequality: for any non-negative $p$ and $q$ and real numbers $a$ and $b$, one has $-(p\, a^2+q\, b^2)\leq 2\,\sqrt{p\,q}\,ab$ and therefore}
$$
-\sum_{i=2}^d \big( a_i^2 f_1(\kappa) +2 a_i b_i f_2(\kappa) + b^2_i f_3(\kappa) \big)
\le -2\,  \big( f_2(\kappa) - \sqrt{f_1(\kappa) f_3(\kappa)}  \big) \sum_{i=2}^d a_i \, b_i.
$$
Using \eqref{eq: coefficients expression} and \eqref{eq: 3rd monotonicty conditon positive rho}
\begin{equation*}
    -\sum_{i=2}^d a_i \, b_i =
    -h\,\langle (I-P_X) S_0, \nabla_{S_0} X \rangle \le  h\,\mu_+ \| \nabla_{S_0} X \|^2.
\end{equation*}
\added{Finally we insert these three bounds into \eqref{p_thm6_inequality_to_check} to obtain}
\begin{equation*}
    \|S^s(1)\|^2 - \|S^s(0)\|^2 \le h \left(h - 2\alpha + 2\mu_+ \left( f_2(\kappa) - \sqrt{f_1(\kappa) f_3(\kappa)} \right) \right) \| \nabla_{S_0} X \|^2,
\end{equation*}
which \deleted{is non-negative} \added{holds} whenever 
\begin{equation*}
  0 <  h \le 2\alpha - 2\mu_+ \left( f_2(\kappa) - \sqrt{f_1(\kappa) f_3(\kappa)} \right).
\end{equation*}
The inequality \eqref{eq: condition on h for positive rho} follows since the function $f_2(\kappa) - \sqrt{f_1(\kappa) f_3(\kappa)}$ \deleted{in} \added{is} non-decreasing for
$0\leq\kappa\leq\pi$.
\end{proof}
\begin{figure}
    \centering   \includegraphics[width=0.5\linewidth]{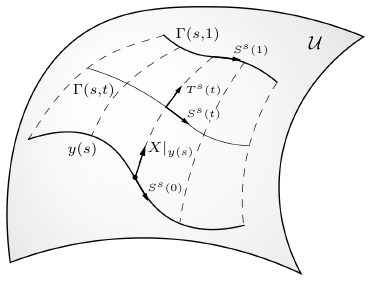}
    \caption{\small Construction for the proof of theorem \ref{theo: GEE non-expansive for positive rho}.}
    \label{fig: construction of the theorem}
\end{figure}
\begin{remark}
    When the curvature tends to zero from above, then $f_1, f_2$ and $f_3$ all tend to zero, and from \eqref{eq: condition on h for positive rho} we retrieve the bound $h\le 2\alpha$, which is the classical
    result of Dahlquist and Jeltsch \cite[Theorem~4.1]{dahlquist1979gdo} for Euclidean space.
%
%
\end{remark}
\begin{remark}
    By combining the $\alpha$-cocoercivity condition with \eqref{eq: 3rd monotonicty conditon positive rho} one finds that $\mu_+ \ge \alpha>0$. This means that the stepsize bound \eqref{eq: condition on h for positive rho} is more restrictive than the corresponding one in Euclidean space.
    $h \le \frac{2\alpha}{1-f_3(\kappa)}$.
    
\end{remark}

\begin{theorem}
	\label{theo: GEE non-expansive for negative rho}
    Let $(\mathcal{M},g)$ be a \deleted{connected} $d$-Riemannian manifold with constant sectional curvature $\rho<0$ and $\mathcal{U}\subset\mathcal{M}$ a geodesically convex subset.
    Let $X\in\mathfrak{X}(\mathcal{M)}$ be $\alpha$-cocoercive on $\mathcal{U}$ for some $\alpha > 0$, and
    $\|X\|\leq C$ when restricted to $\mathcal{U}$.
    Suppose that $\nabla X$ is invertible with a bounded inverse $\|\nabla X^{-1}\|\le \sigma$ uniformly on $\mathcal{U}$.
    Suppose there is a $\mu_-$ such that
\begin{equation}
    \label{eq: 3rd monotonicty conditon negative rho}
    \langle \nabla_{v_p} X, P_X v_p \rangle \ge - \mu_- \| \nabla_{v_p} X \|^2, \quad \forall \, p\in\mathcal{U}, \quad \forall \, v_p \in T_p \mathcal{M}.
\end{equation}
Then for any $x_0, y_0 \in \mathcal{U}$ the Geodesic Euler Method satisfies
    \begin{equation*}
        d \big( \exp_{x_0}(hX|_{x_0}), \exp_{y_0}(hX|_{y_0}) \big) \le d(x_0, y_0)
    \end{equation*}
    whenever the step size $h$ is such that
   \begin{equation}
        \label{eq: condition on h for negative rho}
        0< h\le \frac{2}{1+\sigma^2C|\rho|}\left(\alpha \,\kappa\coth\kappa -\mu_{-}\frac{f_2(\kappa)-\sqrt{f_1(\kappa)f_3(\kappa)}}{1+f_3(\kappa)}\right)
    \end{equation}
 for all $0\le \kappa\le hC\sqrt{|\rho|}$ where $\kappa=h\|X|_{p}\|\sqrt{|\rho|}$ and $p\in\mathcal{U}$, 
  and  where  
   $f_1$, $f_2$, $f_3$ are as in \eqref{f1}-\eqref{f3}. 
\end{theorem}
\begin{figure}
    \centering
    \includegraphics[width=0.95\linewidth]{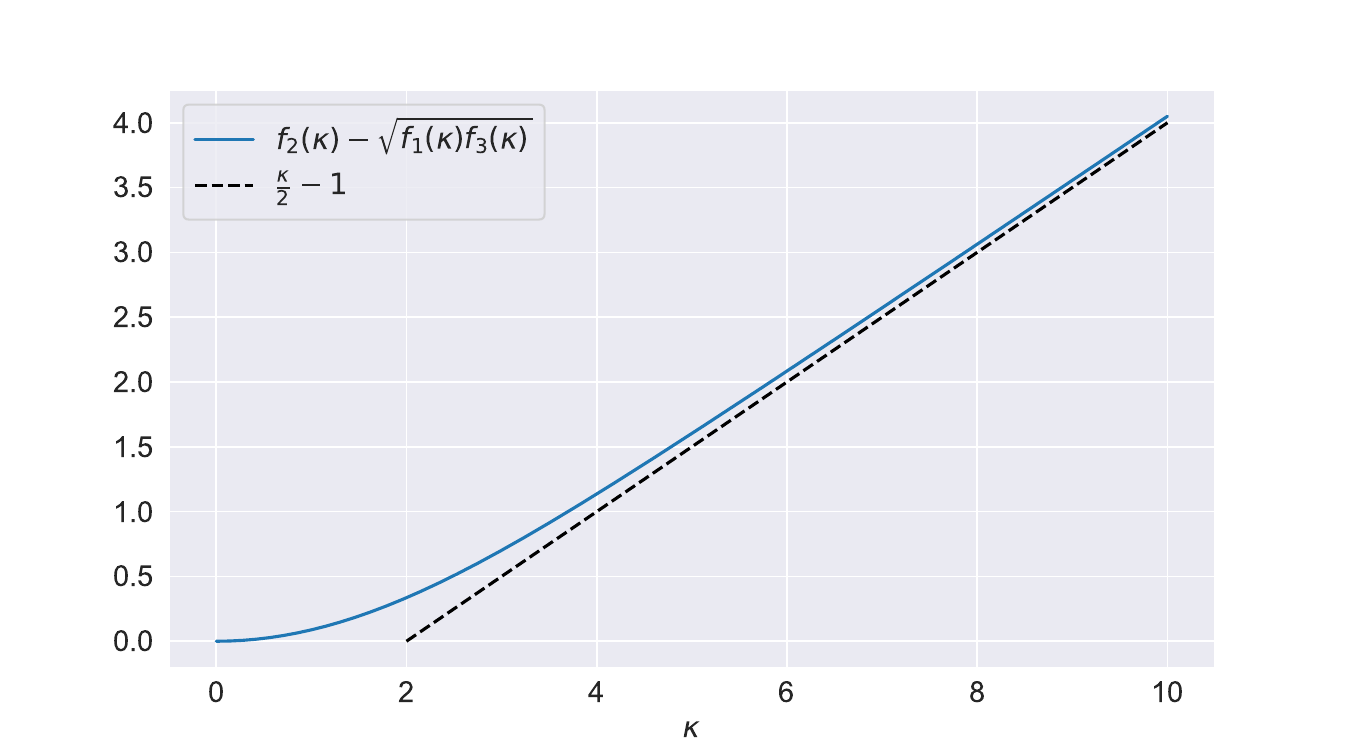}
    \caption{\small \deleted{Visualization} \added{The graph} of the function appearing in Theorem \ref{theo: GEE non-expansive for negative rho}, $f_2(\kappa) - \sqrt{f_1(\kappa) f_3(\kappa)} = \left(\cosh(\kappa)\frac{\sinh(\kappa)}{\kappa}-1\right) - \sinh\kappa\sqrt{\frac{\sinh^2(\kappa)}{\kappa^2}-1}$. For large $\kappa$ it behaves as $\frac{\kappa}{2}-1$.}
    \label{fig: f_plot negative rho}
\end{figure}
\begin{proof}
\added{We follow similar arguments as in the proof of Theorem~\ref{theo: GEE non-expansive for positive rho}, but since the expressions for the Jacobi fields are now different, we also need to modify the bounds we use for the terms in the inequalities. In particular, the sign change in the formula \eqref{nrmdiff} of Lemma~\ref{lemma1} requires the use of the projection $P_{X}$ in \eqref{eq: 3rd monotonicty conditon negative rho} instead of
$I-P_X$ that was used in \eqref{eq: 3rd monotonicty conditon positive rho}.}
We take $x_0, y_0\in\mathcal{U}$, and let $y(s),\, s\in[s_1,s_2]$ be the geodesic satisfying $y(s_1)=x_0$ and $y(s_2)=y_0$.
The variation $\Gamma(s,t)=\exp_{y(s)}(thX|_{y(s)})$ is well defined,
and we let $S^s(t)=\partial_s \Gamma(s,t)\in\mathfrak{X}(\Gamma)$. 
Now, using Lemma~\ref{lemma4}, \eqref{nrmdiff0}, Lemma~\ref{lemma1} and \eqref{nrmdiff}, we write down the following
condition that needs to be satisfied
\begin{equation} \label{p_thm9_inequality_to_check}
 h^2 \| \nabla_{S_0} X \|^2 + 2 h \langle S_0, \nabla_{S_0} X \rangle + \sum_{i=2}^d \big( a_i^2 f_1(\kappa) +2 a_i b_i f_2(\kappa) + b^2_i f_3(\kappa) \big)
\leq 0,
\end{equation}
and, we further split the sum in \eqref{p_thm9_inequality_to_check} as
\begin{align*}
\sum_{i=2}^d &\big( a_i^2 f_1(\kappa) +2 a_i b_i f_2(\kappa) + b^2_i f_3(\kappa) \big) \\
=& \sum_{i=1}^d \big( a_i^2 f_1(\kappa) +2 a_i b_i f_2(\kappa) + b^2_i f_3(\kappa) \big) - a_1^2 f_1(\kappa) -2 a_1 b_1 f_2(\kappa) - b^2_1 f_3(\kappa).
\end{align*}
where $\kappa=\kappa(s)=h\|X|_{y(s)}\| \sqrt{|\rho|}$.
Using that $\|\nabla X^{-1}\|\le \sigma$ we have
\begin{equation*}
 \sum_{i=1}^d a_i^2  =  \|S_0\|^2
\le \|\nabla X^{-1}\nabla X S_0\|^2\le \sigma^2\|\nabla_{S_0}X\|^2,
\end{equation*}
which combined with $\alpha$-cocoercivity gives us the following bound
\begin{align*}
    \sum_{i=1}^d &\big( a_i^2 f_1(\kappa) +2 a_i b_i f_2(\kappa) + b^2_i f_3(\kappa) \big) \\
    &=  f_1(\kappa)\|S_0\|^2 + 2h f_2(\kappa) \langle S_0,\nabla_{S_0}X\rangle + h^2 f_3(\kappa)\|\nabla_{S_0}X\|^2 \\
    &\le \|\nabla_{S_0}X\|^2 \left( \sigma^2 f_1(\kappa) - 2h\alpha f_2(\kappa) + h^2 f_3(\kappa) \right).
\end{align*}
A variant of Young's inequality yields
$$
- a_1^2 f_1(\kappa) - b^2_1 f_3(\kappa) \leq 2  \sqrt{f_1(\kappa)f_3(\kappa)}\, a_1b_1.
$$
Recalling that $a_1 b_1 = h \langle P_X S_0, \nabla_{S_0} X \rangle$, we make use of condition \eqref{eq: 3rd monotonicty conditon negative rho}, 
to obtain 
\begin{align*}
    - a_1^2 f_1(\kappa) -2 a_1 b_1 f_2(\kappa) - b^2_1 f_3(\kappa) & \le -2 \left(f_2(\kappa) - \sqrt{f_1(\kappa)f_3(\kappa)}  \right) a_1 b_1 \\
    &\le 2 h \mu_- \left(f_2(\kappa) - \sqrt{f_1(\kappa)f_3(\kappa)} \right)  \|\nabla_{S_0}X\|^2.
\end{align*}
Combining these bounds, we finally obtain
\begin{multline*}
    \|S^s(1)\|^2 - \|S^s(0)\|^2
    \le  \bigg( h^2 \big(1+f_3(\kappa) \big) - 2h \alpha (1+f_2(\kappa))\\ +2h \mu_- \left(f_2(\kappa) - \sqrt{f_1(\kappa) f_3(\kappa)} ) \right) 
     + \sigma^2 f_1(\kappa) \bigg) \| \nabla_{S_0} X  \|^2
\end{multline*}
which is less than or equal to zero whenever
\begin{equation*}
    h^2 \big(1+f_3(\kappa) \big) - 2h \alpha (1+f_2(\kappa))  + 2h\mu_- \left(f_2(\kappa) - \sqrt{f_1(\kappa) f_3(\kappa)} \right)  + \sigma^2 f_1(\kappa) \le 0.
\end{equation*}
Dividing throughout by $1+f_3(\kappa)\ge 1$ and by $h>0$, and observing that $\frac{f_1(\kappa)}{1+f_{3}(\kappa)}= \kappa^2= h^2\|X|_{y(s)}\|^2|\rho|$
we get
$$h\left(1+\sigma^2\|X|_{y(s)}\|^2|\rho|\right)-2\left( \alpha\,\frac{1+f_2(\kappa) }{1+f_3(\kappa)}-\mu_{-}\frac{f_2(\kappa)-\sqrt{f_1(\kappa)f_3(\kappa)}}{1+f_3(\kappa)}\right)\le 0.$$
Since $\frac{1+f_2(\kappa) }{1+f_3(\kappa)}=\kappa\coth\kappa $, we obtain the bound
$$h\le\frac{2}{1+\sigma^2C^2|\rho|}\left(\alpha\,\kappa \coth\kappa -\mu_{-} \frac{f_2(\kappa)-\sqrt{f_1(\kappa) f_3(\kappa})}{1+f_3(\kappa)}\right),$$
that needs to be satisfied for all $0\le\kappa\le hC\sqrt{|\rho|}$.
\end{proof}
\begin{remark}
\deleted{One should be aware that} The stepsize bounds of Theorems \ref{theo: GEE non-expansive for positive rho}, \ref{theo: GEE non-expansive for negative rho}  
are implicit since the parameter $\kappa$ depends on $h$.
But $\kappa$ also has an interpretation as the distance traveled by an Euler step. 
\added{When we let $h\rightarrow 0$ in \eqref{eq: condition on h for positive rho} and \eqref{eq: condition on h for negative rho} respectively
the two inequalities reduce to 
$$
2\alpha >0,\quad \frac{2}{1+\sigma^2 C |\rho|}>0,
$$
which are both trivially satisfied under the assumptions of the theorem. A continuity argument then guarantees that 
there are nonempty intervals in $h$ for which \eqref{eq: condition on h for positive rho} and \eqref{eq: condition on h for negative rho} hold.}
\end{remark}

We next use the notation
$$\mathcal{K}_p=\text{span}\{\left.X\right|_{p}\},\quad \mathcal{V}_p=\text{range}(\nabla X|_p).$$
The invertibility requirement of Theorem~\ref{theo: GEE non-expansive for negative rho} can be relaxed. In the next proposition we consider the case when $\nabla X|_p$ has a $1$-dimensional kernel $\ker{\nabla X|_p}=\mathcal{K}_p$. 
Then 
$\nabla X|_p$ is invertible on $\mathcal{V}_p$ 
 and we can guarantee the non-expansivity of the GEE method imposing a simpler bound on the step-size $h$. 
We can state the following result. 
\begin{prop}\label{prop: singular nablaX}
    Let $(\mathcal{M},g)$ be a \added{connected} $d$-Riemannian manifold with constant sectional curvature $\rho<0$ and $\mathcal{U}\subseteq \mathcal{M}$ an open geodesically convex set. Let $X\in\mathfrak{X}(\mathcal{M)}$ be an $\alpha$-cocoercive vector field on $\mathcal{U}$ for some $\alpha > 0$,  such that $\ker(\nabla X|_p)=\mathcal{K}_p,$ for all $p\in\mathcal{U}$ and $\|X\|\leq C$ when restricted to $\mathcal{U}$. 
    Suppose there exist a $\sigma > 0$ such that $\|\left.\nabla X\right|_{\mathcal{V}_p}^{-1}\|\le \sigma$ uniformly on $\mathcal{U}$. 
    Then for any $x_0, y_0 \in \mathcal{U}$ the Geodesic Euler Method 
    satisfies
    \begin{equation*}
        d \big( \exp_{x_0}(hX|_{x_0}), \exp_{y_0}(hX|_{y_0}) \big) \le d(x_0, y_0)
    \end{equation*}
    whenever the step size $h$ is such that
   \begin{equation}
   \label{bound_h_prop10}
   0\le h\le \frac{1}{\|X|_{p}\|\sqrt{|\rho|}}\operatorname{arccoth} \left  ( \frac{1+\|X|_{p}\|^2|\rho|\sigma^2}{2\alpha \|X|_{p}\|\sqrt{|\rho|}} \right)
   \end{equation}
    for all $p\in \mathcal{U}$. \deleted{and $\|X|_{p}\|\le C$.} 
\end{prop}
\begin{proof} We proceed as in the proof of Theorems~\ref{theo: GEE non-expansive for positive rho} and~\ref{theo: GEE non-expansive for negative rho} and
we write
\begin{multline*}
    \|S^s(1)\|^2 - \|S^s(0)\|^2 = h^2\| \nabla_{S_0} X\|^2+2h\langle S_0,\nabla_{S_0}X\rangle \\ +f_1(\kappa)\|(I-P_X)S_0\|^2 + 2h f_2(\kappa) \langle (I-P_X)S_0,\nabla_{S_0}X\rangle\\
    + h^2 f_3(\kappa)\|(I-P_X)\nabla_{S_0}X\|^2
\end{multline*}
    where $\kappa = \kappa(s)=h\|X|_{y(s)}\|\sqrt{|\rho|}$ and $y(s)$ is the geodesic defined on $[s_1,s_2]$ with $y(s_1)=x_0$ and $y(s_2)=y_0$.
   If $S_0\in \mathcal{K}_{y(s)}=\ker(\nabla X|_{y(s)})$,  we get
   $$ \|S^s(1)\|^2 - \|S^s(0)\|^2=f_1(\kappa)\|(I-P_X)S_0\|^2=0,\qquad \forall h >0$$
   where we have used that $P_XS_0=S_0$ and $(I-P_X)S_0=0$.
 This grants non-expansivity  for all $h>0$, (locally in $s$).
 
For $S_0\notin \mathcal{K}_{y(s)}$,  $(I-P_X)S_0\neq 0$ and
$\nabla_{S_0} X =  \nabla_{(I-P_X)S_0}X $.
We use the restriction of $\nabla X|_{y(s)}$ to $\text{range}(I-P_{X})$  
%
  and the bound $\|\left. \nabla X\right|_{\mathcal{V}_{y(s)}}^{-1}\|\le \sigma$ to obtain  
$$\|(I-P_X)S_0\|^2= \|\left. \nabla X\right|_{\mathcal{V}_{y(s)}}^{-1}\nabla_{(I-P_X)S_0} X \|^2 \le \sigma^2\|\nabla_{(I-P_X)S_0}X\|^2=\sigma^2 \|\nabla_{S_0}X\|^2.$$
Since
$$ \langle (I-P_X)S_0,\nabla_{S_0}X\rangle=\langle (I-P_X)S_0,\nabla_{(I-P_X)S_0}X\rangle\le -\alpha \|\nabla_{(I-P_X)S_0}X\|^2,$$ we get
$$ \|S^s(1)\|^2 - \|S^s(0)\|^2\le (h^2(1+f_3(\kappa))-2h\alpha(1+f_2(\kappa))+f_1(\kappa)\sigma^2)\|\nabla_{S_0}X\|^2.$$
   Following similar arguments as in the proof of Theorem~\ref{theo: GEE non-expansive for negative rho}, we obtain the condition
   $$
   h\le\frac{2\alpha}{1+\|X|_{y(s)}\|^2|\rho|\sigma^2}\,\frac{1+f_2(\kappa)}{1+f_3(\kappa)}= \frac{2\alpha\kappa\coth\kappa}{1+\|X|_{y(s)}\|^2|\rho|\sigma^2}.
   $$
   Replacing $\kappa=h \|X|_{y(s)}\| \sqrt{|\rho|}$ and dividing by $h$ on both sides, we get
   $$
   \coth\kappa \geq \frac{1+\|X|_{y(s)} \|^2 |\rho| \sigma^2}{2\alpha\| X|_{y(s)} \| \sqrt{|\rho|}} \ge 1
   $$
where the last inequality relies on the fact that \added{$\sigma\geq \| \nabla X|^{-1} \| \ge 1 / \| \nabla X \| \ge \alpha$ (see Remark \ref{remark: norm of nabla X})}. We can therefore take the inverse of $\coth\kappa$, which is a decreasing function of $\kappa$, and then solve for $h$ to obtain the stated result.
\end{proof}

\section{Examples} \label{sec:examples}
In this section we demonstrate the presented theory for three different examples. The choices made in the experiments need some explanation.
Common to all three examples is that the chosen vector field is a sum of a simple Killing vector field and a constant $\epsilon$ times a gradient vector field.
Killing vector fields are not cocoercive, hence we always choose $\epsilon\neq 0$, and by convention $\epsilon>0$.
In the formulas derived and the numerical tests, we consider the infinitesimal case where the points $x_0$ and $y_0$ are close together so that the geodesically convex set $\mathcal{U}$ can be chosen arbitrarily small, and we consider the conditions on the vector field only at one point, say $x_0$, but the direction of the geodesic between $x_0$ and a nearby point $y_0$ can have any tangent direction.

\subsection{An example on $S^2$ }
Let $\mathcal{M}=S^2$, which has sectional curvature $\rho=1$. In spherical coordinates $(\phi, \theta)$, we consider the vector field with real parameter $\epsilon$
\begin{equation}
    X = \epsilon \cos(\phi) \partial_{\phi} + \partial_{\theta},
\end{equation}
where $\phi\in[-\frac{\pi}{2}, \frac{\pi}{2}]$ is the elevation with respect to the $xy$-plane, and $\theta\in[0,2\pi)$ is the azimuth. Then one finds that the metric is $\dd\phi^2 + \cos^2(\phi) \,\dd\theta^2$, and the non-zero Christoffel symbols $\Gamma^k_{ij}$ are
\begin{equation*}
    \Gamma^1_{2,2}=\sin(\phi)\cos(\phi),\qquad  \Gamma^2_{2,1} = \Gamma^2_{1,2} =-\tan(\phi).
\end{equation*}
We compute the $2\times 2$ connection matrix $\nabla X$ in coordinates,  
\begin{equation*}
    \nabla X = 
    \begin{bmatrix}
        -\epsilon\sin(\phi) &  \cos(\phi)\sin(\phi)    \\
        -\tan (\phi) &      -\epsilon\sin(\phi) 
    \end{bmatrix}.
\end{equation*}
The condition for $X$ to generate a non-expansive flow is that the logarithmic $g$-norm  of $\nabla X$ is non-positive \cite{arnold24bso, dekker1984, söderlind2024logarithmic}. This quantity, defined in (\ref{eq: logarithmic g-norm}), is found to be
\begin{equation*}
    \mu_g(\nabla X) = -\epsilon\sin(\phi)   
\end{equation*}
which makes $X$ contractive on the upper hemisphere if $\epsilon>0$, and on the lower hemisphere if $\epsilon < 0$. A basis for $T_p S^2$, $p\in\mathcal{M}$, is
\begin{align*}
    e_1 &= \frac{X|_p}{\|X|_p \|} = \frac{1}{\sqrt{1+\epsilon^2}} \left(\epsilon \, \partial_{\phi} + \frac{1}{\cos(\phi)} \partial_{\theta} \right), \\
    e_2 &= \frac{1}{\sqrt{1+\epsilon^2}} \left(- \partial_{\phi} + \frac{\epsilon}{\cos(\phi)} \partial_{\theta} \right),
\end{align*}
and we can write any $v_p\in T_p S^2$ as $v_p = \xi e_1 + \eta e_2$, $\xi, \eta \in \mathbb{R}$. As we are only interested in the direction given by $v_p$, we consider unit norm vectors. For $p$ in the upper hemisphere one can explicitly compute the cocoercivity constant $\alpha$ in (\ref{eq:cocoercivity-condition})
\begin{equation*}
    \alpha \le \frac{\epsilon}{(1+\epsilon^2)\sin(\phi)},
\end{equation*}
and $\mu_+$ in (\ref{eq: 3rd monotonicty conditon positive rho})
\begin{equation*}
    \mu_+ \ge \left(1 + \frac{\sqrt{1+\epsilon^2}}{2\epsilon\left(1+\epsilon^2+\epsilon\sqrt{1+\epsilon^2}\right)} \right) \frac{\epsilon}{(1+\epsilon^2)\sin(\phi)}.
\end{equation*}
Replacing the above constants and $\kappa = h \sqrt{1+\epsilon^2}\cos(\phi)$ in condition \eqref{eq: condition on h for negative rho}, one finds the smallest step size such that the inequality is violated.

\paragraph{\bf Numerical results.}
Figure \ref{fig: testS2_different_epsilon} shows some numerical results for our test case on $S^2$. The plots show the bound on the stepsize computed numerically (the smallest $h$ such that $\|S^s(1)\|-\|S^s(0)\| > 0$), versus the one given by (\ref{eq: condition on h for positive rho}), for different initial elevation angles $\phi_0$ and different values of the parameter $\epsilon$. For this test case, Theorem \ref{theo: GEE non-expansive for positive rho} gives a condition on the stepsize which guarantees the non-expansivity of the (\ref{eq: GEE}) method. Such condition gets closer to the numerical one as the initial point gets further from the equator.
\begin{figure}
    \centering
    \includegraphics[width=0.32\linewidth]{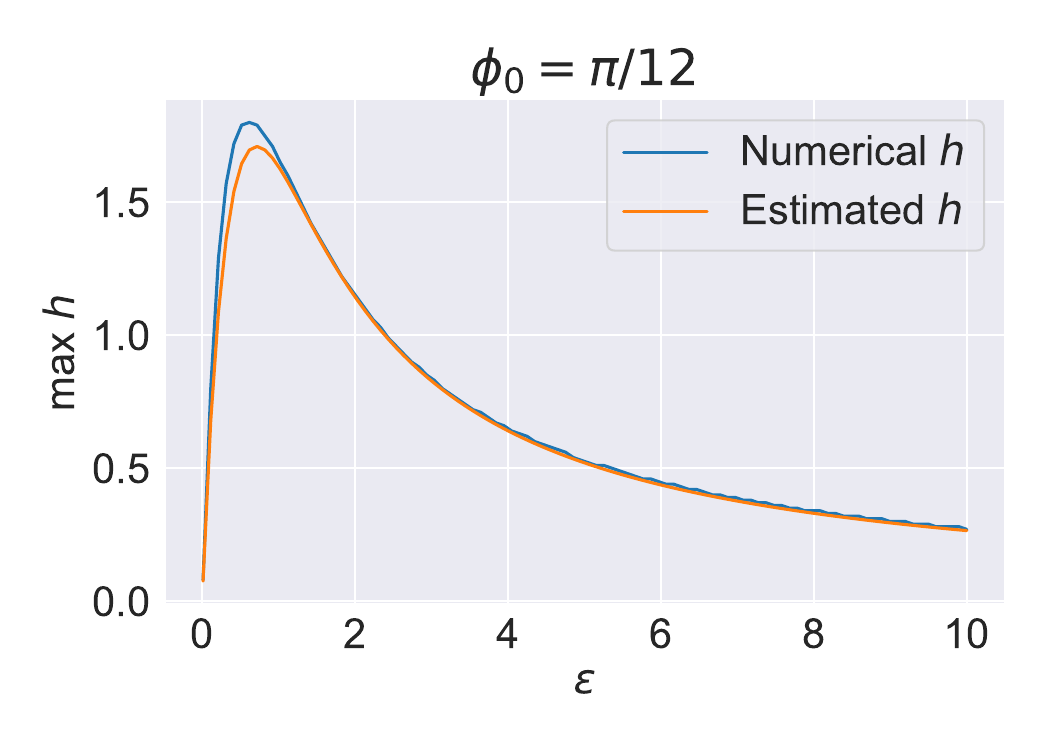}
    \includegraphics[width=0.32\linewidth]{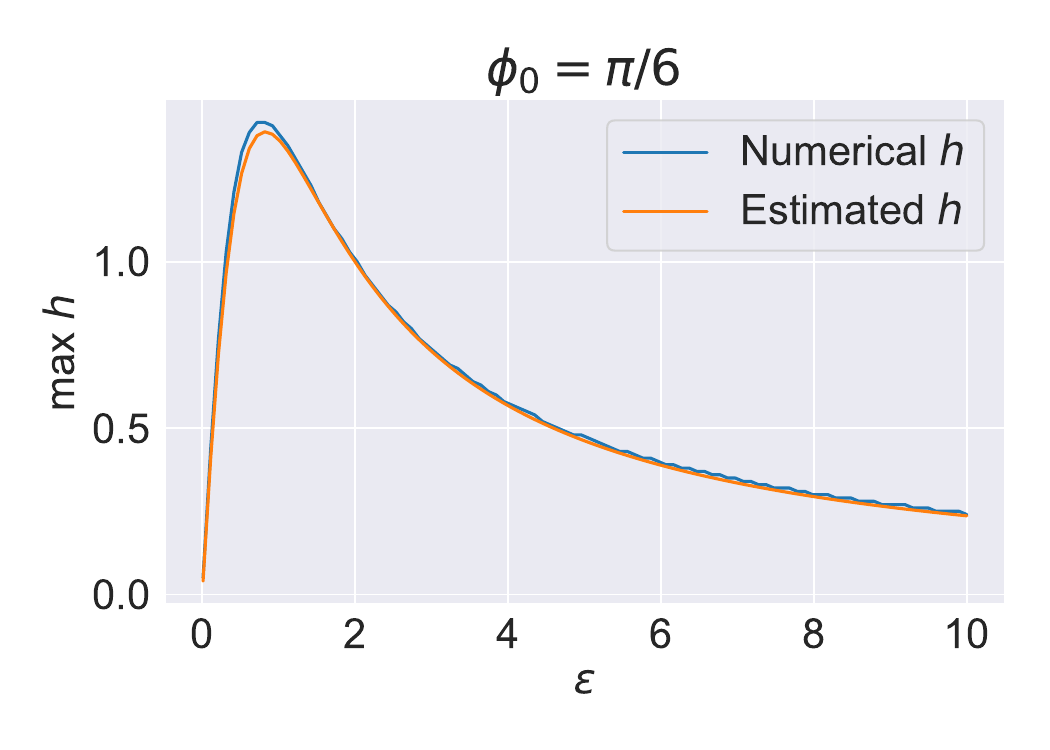}
    \includegraphics[width=0.32\linewidth]{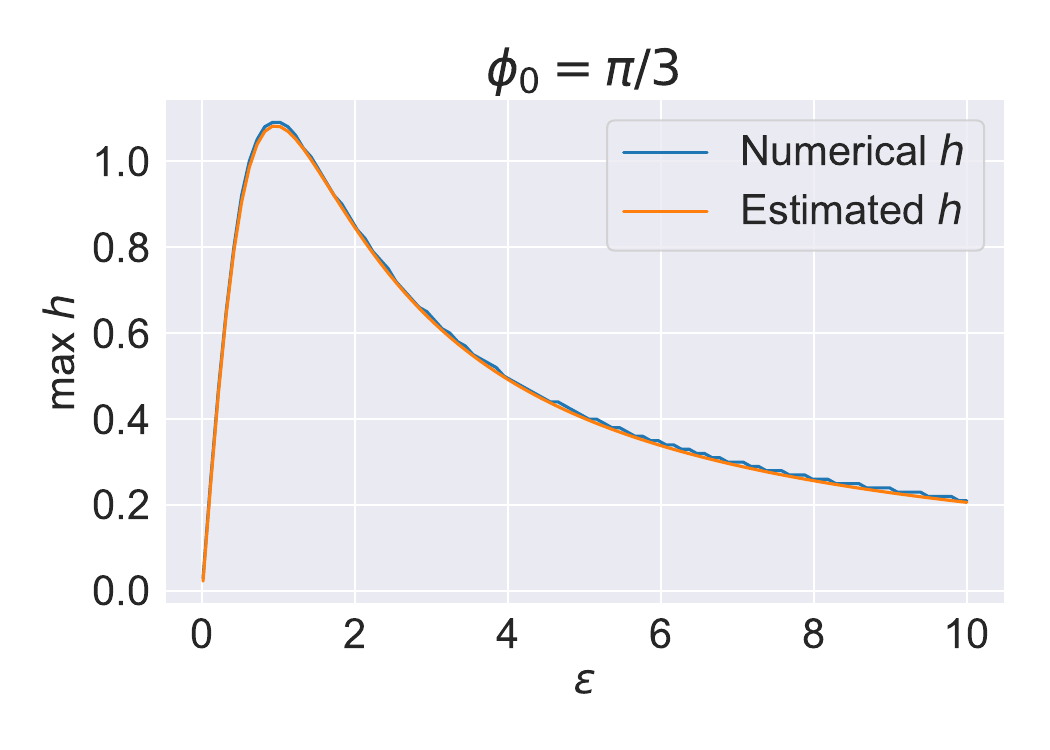}
    \caption{\small The plots show the largest stepsize $h$ found numerically versus the one given by condition \eqref{eq: condition on h for positive rho} in Theorem \ref{theo: GEE non-expansive for positive rho}, for different initial elevation angles $\phi_0$ and different values of the parameter $\epsilon$.}
    \label{fig: testS2_different_epsilon}
\end{figure}

\subsection{An example on $\mathbb{H}^2$}
Let $\mathcal{M}$ be the hyperbolic half plane $\mathbb{H}^2 = \{ (x,y)\in\mathbb{R}^2 \, : \, y > 0 \}$, which has sectional curvature $\rho=-1$. In cartesian coordinates, we consider the vector field with real parameter $\epsilon$
\begin{equation}
    X = \partial_{x} + \epsilon \partial_{y}.
\end{equation}
The metric is found to be $\frac{1}{y^2} \,\dd x^2 + \frac{1}{y^2} \, \dd y^2$, and the non-zero Christoffel symbols $\Gamma^k_{ij}$ are
\begin{equation*}
    \Gamma^1_{1,2}=\Gamma^1_{2,1}= -\frac{1}{y},\qquad  \Gamma^2_{1,1} = \frac{1}{y} = - \Gamma^2_{2,2}.
\end{equation*}
The $2\times 2$ connection matrix $\nabla X$ and its inverse in coordinates read  
\begin{equation*}
    \nabla X = \frac{1}{y}
    \begin{bmatrix}
        -\epsilon &  -1    \\
        1 &      -\epsilon
    \end{bmatrix}, \qquad \nabla X^{-1} = \frac{y}{\epsilon^2+1}
    \begin{bmatrix}
        -\epsilon &  1    \\
        -1 &      -\epsilon
    \end{bmatrix}.
\end{equation*}
Again, the flow of $X$ is non-expansive if the logarithmic $g$-norm  of $\nabla X$ is non-positive \cite{arnold24bso, dekker1984, söderlind2024logarithmic}. According to (\ref{eq: logarithmic g-norm}), we find
\begin{equation*}
    \mu_g(\nabla X) = -\epsilon,  
\end{equation*}
which makes $X$ contractive whenever $\epsilon>0$. A basis for $T_p \mathbb{H}^2$, $p=(x,y)\in\mathcal{M}$, is
\begin{align*}
    e_1 &= \frac{X|_p}{\|X|_p \|} = \frac{y}{\sqrt{1+\epsilon^2}} \big( \partial_{x} + \epsilon \partial_{y} \big), \\
    e_2 &= \frac{y}{\sqrt{1+\epsilon^2}} \big( \epsilon \partial_{x} - \partial_{y} \big),
\end{align*}
and we can write any $v_p\in T_p \mathbb{H}^2$ as $v_p = \xi e_1 + \eta e_2$, $\xi, \eta \in \mathbb{R}$. It suffices to consider only unit norm vectors $v_p$. We explicitly compute the cocoercivity constant $\alpha$ in (\ref{eq:cocoercivity-condition})
\begin{equation*}
    \alpha \le \frac{\epsilon y}{(1+\epsilon^2)}
\end{equation*}
the uniform bound on $\|\nabla X^{-1}\|$
\begin{equation*}
    \sigma \ge \frac{y}{\sqrt{1+\epsilon^2}}
\end{equation*}
and $\mu_-$ in (\ref{eq: 3rd monotonicty conditon negative rho}) 
\begin{equation*}
    \mu_- \ge \left(\frac{\sqrt{1+\epsilon^2}}{2\epsilon} + \frac{1}{2} \right) \frac{\epsilon y}{(1+\epsilon^2)}.
\end{equation*}
Replacing the above constants and $\kappa = h \frac{\sqrt{1+\epsilon^2}}{y}$ in condition \eqref{eq: condition on h for negative rho} one can numerically solve for $h$ to find the smallest positive step size for which the inequality is violated.

\paragraph{\bf Numerical results.}
Figure \ref{fig: testH2_different_epsilon} shows some numerical results for our test case on $\mathbb{H}^2$. The plots show the bound on the stepsize computed numerically (the smallest $h$ such that $\|S^s(1)\|-\|S^s(0)\| >  0$), versus the one obtained by solving numerically condition \eqref{eq: condition on h for negative rho}, for different initial $y_0$ and different values of the parameter $\epsilon$. For this test case, Theorem \ref{theo: GEE non-expansive for negative rho} gives a condition on the step size which guarantees the non-expansivity of the (\ref{eq: GEE}) method.
\begin{figure}
    \centering
    \includegraphics[width=0.32\linewidth]{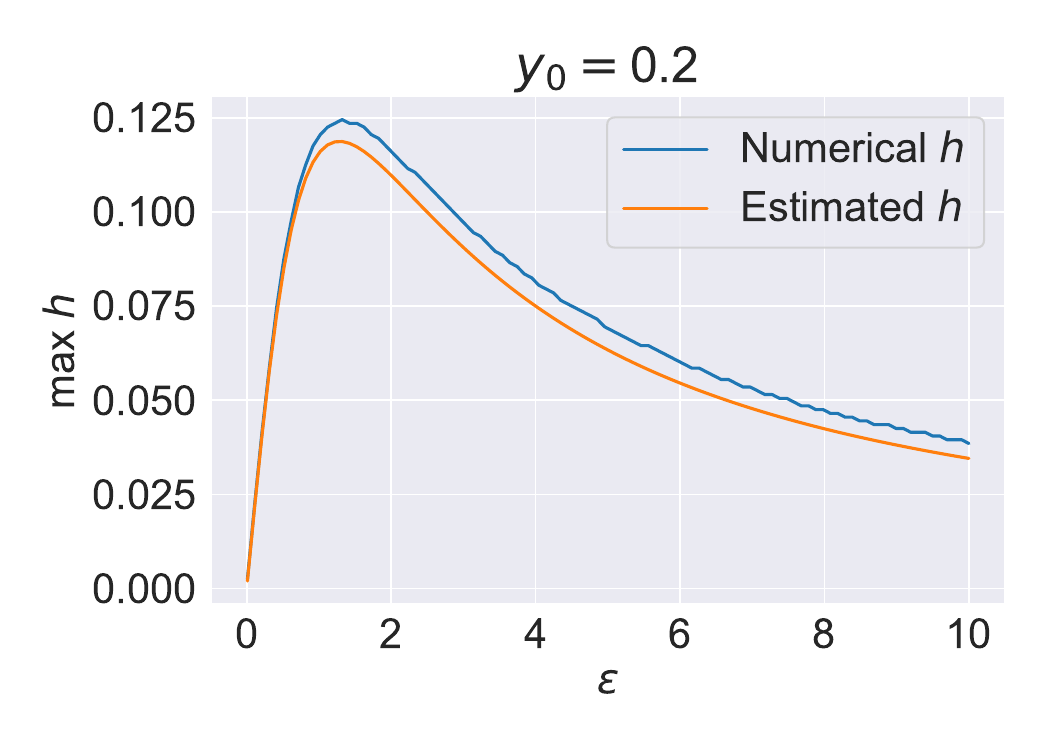}
    \includegraphics[width=0.32\linewidth]{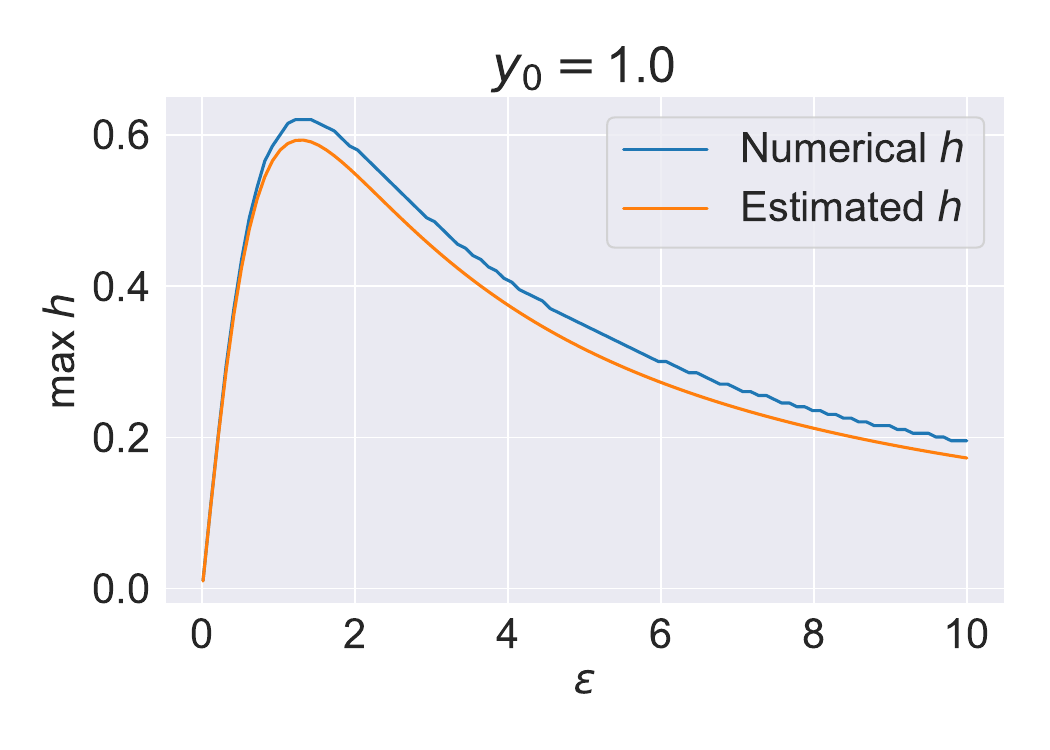}
    \includegraphics[width=0.32\linewidth]{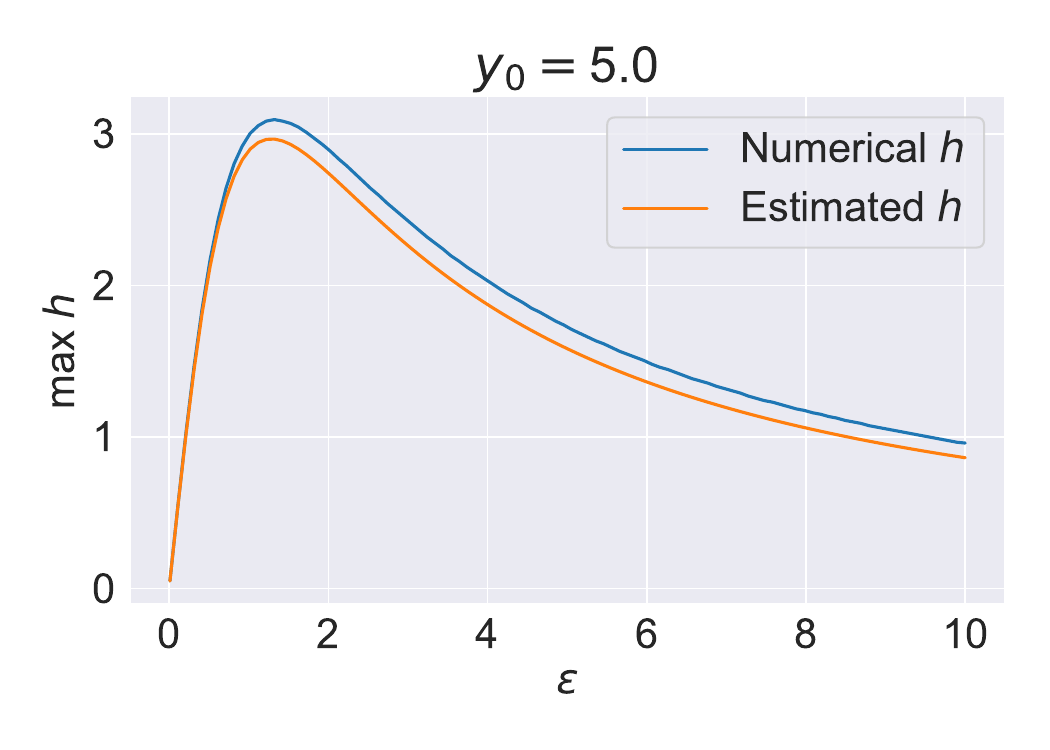}
    \caption{\small The plots show the largest stepsize $h$ found numerically versus the estimated one given by solving numerically condition \eqref{eq: condition on h for negative rho} in Theorem \ref{theo: GEE non-expansive for negative rho}, for different initial $y_0$ and different values of the parameter $\epsilon$.}
    \label{fig: testH2_different_epsilon}
\end{figure}

In order to illustrate Proposition~\ref{prop: singular nablaX}, we consider the vector field $X=y\partial_y$ on $\mathbb{H}^2$.
The basis vectors are $e_1=X=y\partial_y$, $e_2=-y\partial_x$.
We get the singular connection matrix
$$
\nabla X = \left[
\begin{array}{cc}
  -1   &  0 \\
  0   &  0
\end{array}
\right],
$$
 With $S_0=\xi e_1 + \eta e_2$ we get
$$
\langle\nabla_{S_0} X, S_0\rangle = -\eta^2,\quad \|\nabla_{S_0} X \|^2 = \eta^2,
$$
so we can take the cocoercivity constant to be $\alpha=1$. Then $\mathcal{V}_p=\text{Range}(\nabla X)=\mbox{span}(e_2)$, and we obtain
$\sigma=1$ as well as $C=\sqrt{\langle X, X\rangle}=1$ and $\kappa=h$. Proposition~\ref{prop: singular nablaX} now tells us that 
$d(x_1,y_1)\leq d(x_0,y_0)$ if
$$
h \leq  h\coth(h)
$$
which holds for any $h>0$.

\subsection{An example on $S^3$}
Let $\mathcal{M}=S^3$, which has sectional curvature $\rho=1$. In spherical coordinates $(\psi,\theta,\varphi)$, we consider the vector field with real parameter $\epsilon$
\begin{equation}
    X = - \epsilon \sin(\psi) \partial_{\psi} + \partial_{\varphi},
\end{equation}
where $\psi, \theta, \in[0, \pi]$ and $\varphi\in[0,2\pi)$. The round metric is $\dd\psi^2 + \sin^2(\psi) \,\dd\theta^2 + \sin^2(\psi) \sin^2(\theta) \,\dd\varphi^2$, and the non-zero Christoffel symbols $\Gamma^k_{ij}$ are
\begin{align*}
    \Gamma^1_{2,2}&=-\cos(\psi) \sin(\psi), 
    & \Gamma^1_{3,3} &= \cos(\psi) \sin(\psi) \sin^2(\theta), \\
    \Gamma^2_{1,2}&=\Gamma^2_{2,1}= \cot(\psi), 
    & \Gamma^2_{3,3} &= \cos(\theta) \sin(\theta), \\
     \Gamma^3_{1,3}&=\Gamma^2_{3,1}= \cot(\psi), 
    & \Gamma^3_{2,3} &=\Gamma^3_{3,2}= \cot(\theta)
\end{align*}
We compute the $3\times 3$ connection matrix $\nabla X$ in coordinates,  
\begin{equation*}
    \nabla X = 
    \begin{bmatrix}
        -\epsilon\cos(\psi) & 0 &  -\cos(\psi)\sin(\psi)\sin^2(\theta)    \\
        0 & -\epsilon\cos(\psi) & -\cos(\theta)\sin(\theta) \\
        \cot(\psi) & \cot(\theta) & -\epsilon\cos(\psi)
    \end{bmatrix}.
\end{equation*}
The condition for $X$ to generate a non-expansive flow is that the logarithmic $g$-norm  of $\nabla X$ is non-positive \cite{arnold24bso, dekker1984, söderlind2024logarithmic}. This quantity, defined in (\ref{eq: logarithmic g-norm}), is found to be
\begin{equation*}
    \mu_g(\nabla X) = -\epsilon\cos(\psi)   
\end{equation*}
which makes $X$ contractive if $\epsilon>0$ when $\psi\in(0,\frac{\pi}{2})$, and if $\epsilon<0$ when $\psi\in(\frac{\pi}{2},\pi)$. A basis for $T_p S^3$, $p\in\mathcal{M}$, is
\begin{align*}
    e_1 &= \frac{X|_p}{\|X|_p \|} = \frac{1}{\sqrt{\epsilon^2+\sin^2(\theta)}} \left(\epsilon \, \partial_{\psi} + \frac{1}{\sin(\theta)} \partial_{\varphi} \right), \\
    e_2 &= \frac{1}{\sin(\psi)} \partial_{\theta} \\
    e_2 &= \frac{1}{\sqrt{\epsilon^2+\sin^2(\theta)}} \left(\sin(\theta) \, \partial_{\psi} + \frac{\epsilon}{\sin(\psi)\sin(\theta)} \partial_{\varphi} \right),
\end{align*}
and we can write any $v_p\in T_p S^3$ as $v_p = \xi e_1 + \eta e_2 + \zeta e_3$, $\xi, \eta, \zeta \in \mathbb{R}$. As we are only interested in the direction given by $v_p$, we consider unit norm vectors. For $p$ so that $\psi\in(0,\frac{\pi}{2})$, one can explicitly compute the cocoercivity constant $\alpha$ in (\ref{eq:cocoercivity-condition})
\begin{equation*}
    \alpha \le \frac{\epsilon\cos(\psi)}{\cos^2(\psi)(\epsilon^2+\sin^2(\theta)) + \cos^2(\theta)},
\end{equation*}
and $\mu_+$ in (\ref{eq: 3rd monotonicty conditon positive rho}) is computed numerically using the formula
\begin{equation*}
    \mu_+ =\lambda_{\max}\left(\frac{M+M^T}{2}\right),\quad M=-g^{1/2}(I-P_{X})\nabla X^{-1}g^{-1/2}
    \end{equation*}
    and $\lambda_{\max}(A)$ denotes the maximum eigenvalue $A$.
Replacing the above constants and $\kappa = h \sin(\psi) \sqrt{\epsilon^2+\sin^2(\theta)}$ in condition \eqref{eq: condition on h for negative rho}, one finds the smallest positive step size such that the inequality is violated.

\paragraph{\bf Numerical results.}
Figure \ref{fig: testS3_different_epsilon} shows some numerical results for our test case on $S^3$. The plots show the bound on the stepsize computed numerically (the smallest $h$ such that $\|S^s(1)\|-\|S^s(0)\| > 0$), versus the one given by (\ref{eq: condition on h for positive rho}), for different initial angles $\psi_0$ and $\theta_0$, and different values of the parameter $\epsilon$. For this test case, Theorem \ref{theo: GEE non-expansive for positive rho} gives a condition on the stepsize which guarantees the non-expansivity of the (\ref{eq: GEE}) method.
\begin{figure}
    \centering
    \includegraphics[width=0.32\linewidth]{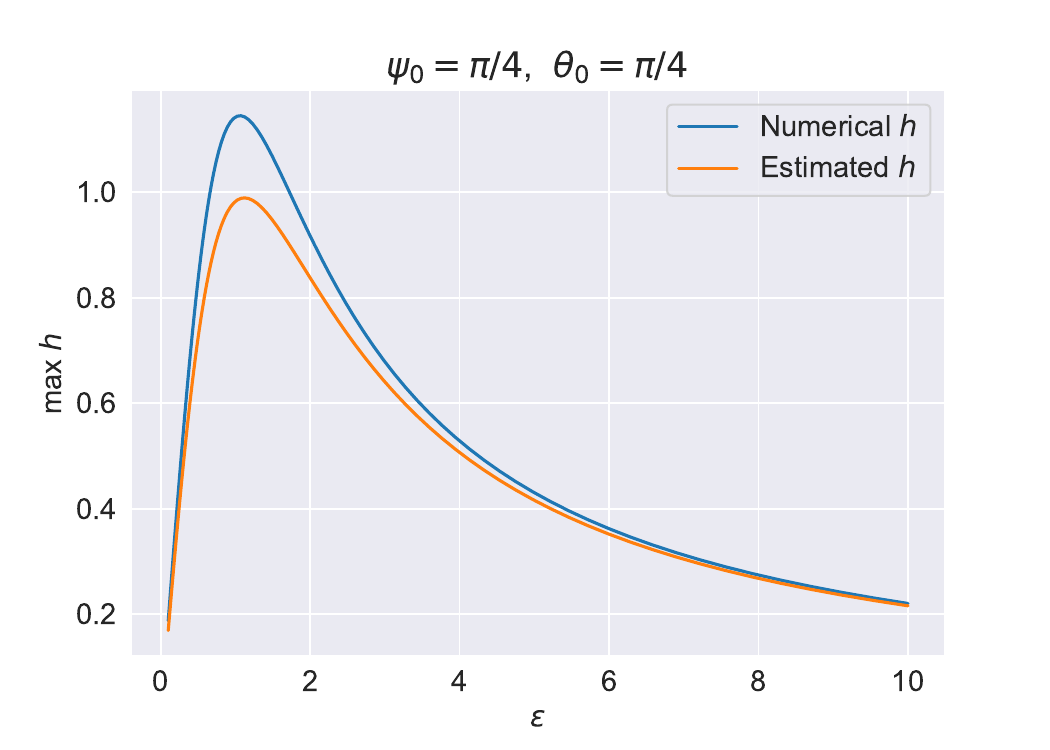}
    \includegraphics[width=0.32\linewidth]{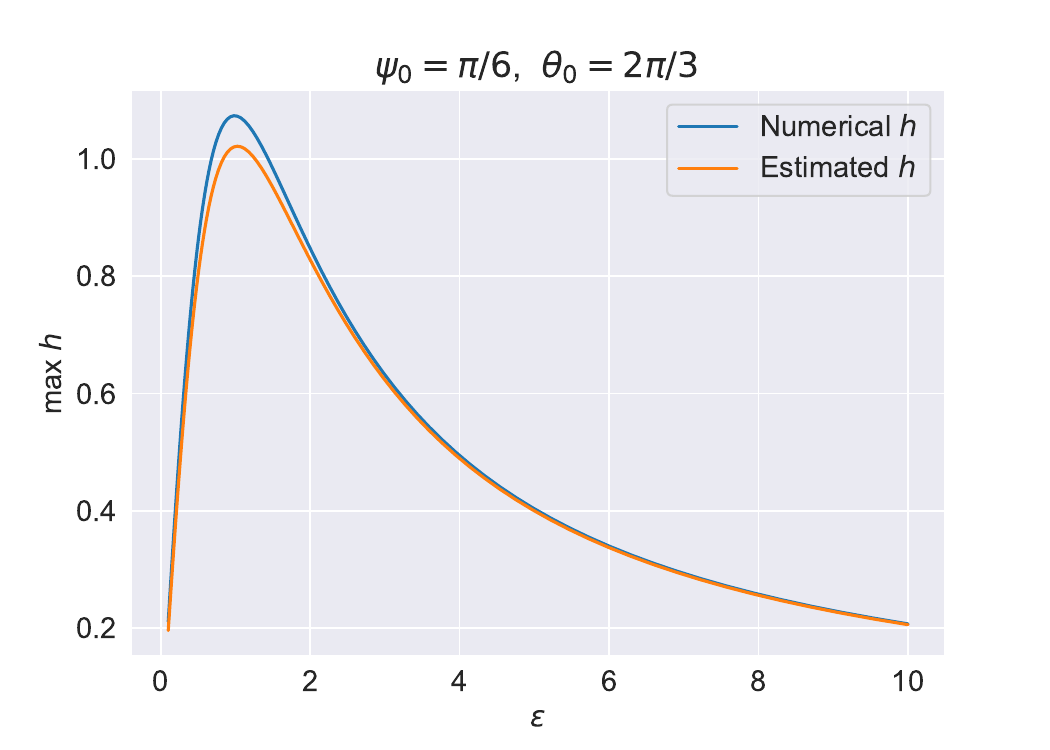}
    \includegraphics[width=0.32\linewidth]{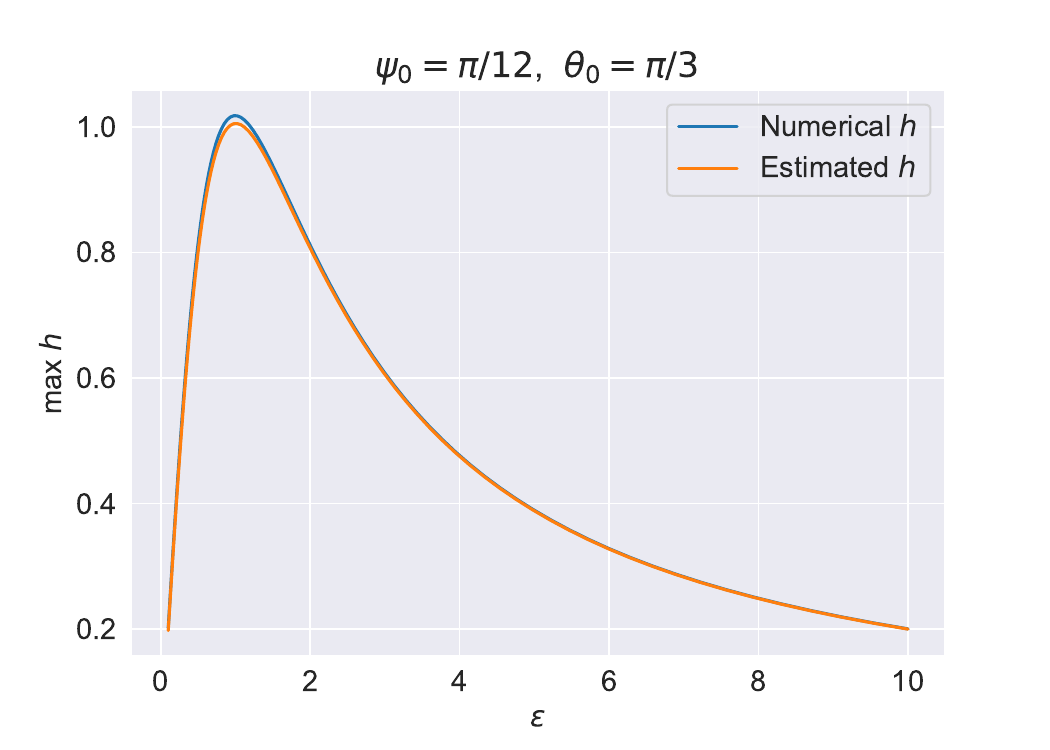}
    \caption{\small The plots show the largest stepsize $h$ found numerically versus the one given by condition \eqref{eq: condition on h for positive rho} in Theorem \ref{theo: GEE non-expansive for positive rho}, for different initial initial angles $\psi_0$ and $\theta_0$, and different values of the parameter $\epsilon$.}
    \label{fig: testS3_different_epsilon}
\end{figure}
 
 \section{Conclusion and future work}
 We have presented a framework for analyzing conditional stability of numerical integrators on Riemannian manifolds.
 The results are obtained for cocoercive ODE vector fields, and we are not aware of results of this kind in the literature that is not assuming a strong monotonicity condition on the vector field.
 
 In this paper we have chosen to focus on the simplest natural numerical integrator on a Riemannian manifold, namely the Geodesic Euler Method.
 The presented work is applicable to manifolds of constant sectional curvature where precise results have been possible to obtain.
 One may however consider these manifolds as model spaces for more general types of Riemannian manifolds, but further investigations are necessary in order to conclude whether the obtained results are representative, e.g. locally, also for the case of non-constant sectional curvatures.

We have seen that stability, or more precisely non-expansivity of numerical schemes on general Riemannian manifold is fundamentally different from that in Euclidean space in some ways. A crucial parameter is what we have denoted $\kappa=h\|X\|\sqrt{|\rho|}$ which is a measure for the distance traveled over a step by the Euler method. This parameter does not affect per se the stability behavior in Euclidean space, but it does add a limitation to the step size in the Riemannian case. The effect is less severe in positively than in negatively curved spaces.
 
 It would be of interest to generalize the results to other types of numerical integrators than the GEE method. It gives some technical advantages to consider methods which are intrinsic, i.e. based on natural operations such as geodesics and parallel transport. But there are also good reasons to look at other types of integrators. One is that geodesics and parallel transport may be challenging to compute in practice for many manifolds, another is that there is a rich variety of other types of integrators available for various types of differentiable manifolds, such as Lie group integrators and retraction based integrators.

\section*{Acknowledgements}
This work is supported by the Horizon Europe, MSCA-SE project 101131557 (REMODEL). A large part of the work was done while the authors were hosted by partners of the REMODEL project. We would like to thank our hosts Takaharu Yaguchi (Kobe University), Lars Ruthotto (Emory University) and Carola Sch\"{o}nlieb (Cambridge University) for accommodating us and for engaging in fruitful discussions. The authors would also like to acknowledge PhD students Nicky van den Berg and Gijs Bellaard from TU Eindhoven for inspiring discussions about this work during our overlapping stay at Emory University in the autumn of 2024.

\bibliographystyle{plain}
\bibliography{GEE}

\appendix

\section{Remarks about the cocoercivity condition}

\subsection{An existence result}
We first give a result which shows under which circumstances the constant $\alpha$ in \eqref{eq:cocoercivity-condition} can be defined.
\begin{prop}\label{prop:cocoexist}
Let $(M,g)$ be a Riemannian manifold and $X$ a smooth vector field with a bounded $\nabla X$ on $\mathcal U\subset M$. 
Suppose that for every $p\in \mathcal{U}$, either
\begin{itemize}
\item $\nabla X|_p$ is invertible with a bounded inverse on $\mathcal{U}$, or
\item $\nabla X|_p$ is singular and $\text{Range}(\nabla X|_p)=\text{Range}(\nabla X|_p^*)$ where $\nabla X|_p^*$ is the adjoint of $\nabla X|_p$.
The restriction of $\nabla X|_p$ to $\text{Range}(\nabla X|_p^*)$ has a bounded inverse.
\end{itemize}
Then $\alpha\in\mathbb{R}$ exists such that
\begin{equation} \label{eq: alphadef}
\langle v_p, \nabla_{v_p}X\rangle \leq -\alpha_p \|\nabla_{v_p} X\|^2,\quad \alpha=\inf_{p\in\mathcal{U}} \alpha_p.
\end{equation}
\end{prop}

\begin{proof}
In the first case, we set $v_p=(\nabla X|_p)^{-1}w_p$ and apply the Cauchy-Schwarz inequality to obtain
$$
\langle \nabla_{v_p} X, v_p\rangle = \langle w_p, (\nabla X|_p)^{-1}w_p\rangle \leq \|(\nabla X|_p)^{-1}\|\,\|w_p\|^2
= \|(\nabla X|_p)^{-1}\|\, \|\nabla_{v_p} X|\|^2
$$
so $-\sup_{p\in\mathcal{U}}\|(\nabla X|_p)^{-1}\|$ can be taken as a lower bound for $\alpha_p$.

In the second case, we choose $v_p = r_p + \beta n_p$ where $r_p\in \text{Range}(\nabla X|_p^*)$, $\beta\in\mathbb{R}$ and $0\neq n_p\in \ker (\nabla X|_p)$.
Then, the condition \eqref{eq: alphadef} reads
$$
\langle \nabla_{r_p} X, r_p\rangle + \beta\langle \nabla_{r_p} X, n_p\rangle \leq -\alpha\|\nabla_{r_p}X\|^2
$$
where $\beta$ is arbitrary. Therefore, it is necessary that $\langle \nabla_{r_p} X, n_p\rangle=0$ 
which is true for all choices of $r_p$ and $n_p$ if and only if $\nabla_{r_p}X\in \text{Range}(\nabla X|^*_p)$. By definition, the set of all $\nabla_{r_p} X|_p$ is precisely $\text{Range}(\nabla X|_p)$, so we conclude that one must have $\text{Range}(\nabla X|_p)=\text{Range}(\nabla X|_p^*)$.
The proof is completed in the same way as in the non-singular case, by restricting $\nabla X|_p$ to $\text{Range}(\nabla X|_p^*)$.
\end{proof}
Note that this proposition does not ensure that the vector field $X$ is cocoercive since it might happen (as in the bound of the proof) that $\alpha<0$.

\subsection{A coordinate formula for the cocoercivity constant.}
We give a concrete formula for the cocoercivity constant when it exists. Suppose we have introduced local coordinates $x_1,\ldots,x_d$.
Now $g$ is to be interpreted as a $d\times d$ matrix with elements $g_{ij}=\langle \partial x_i, \partial x_j\rangle$.
The matrix of the operator $\nabla X|_p$ is denoted $\mathcal{A}_{X,p}$. We assume that we can compute its reduced singular value decomposition
$\mathcal{A}_{X,p}=U\Sigma V^T$ where $U, V\in\mathbb{R}^{d\times r}$ have orthonormal columns and $\Sigma\in\mathbb{R}^{r\times r}$ is diagonal and positive definite and where $r$ is the rank of $\mathcal{A}_{X,p}$.
Thanks to Proposition~\ref{prop:cocoexist} we can write $U=VQ$ where $Q$ is orthogonal and $r\times r$, i.e. $Q=V^TU$.
We are interested in finding the largest possible $\alpha_p$ such that
$$
v_p^T g \mathcal{A}_{X,p}v_p \leq -\alpha_p v_p^T \mathcal{A}_{X,p}^T g \mathcal{A}_{X,p} v_p,\quad\forall v_p\in\mathbb{R}^d
$$
It is easy to show that $-\alpha_p$ can be chosen as the largest eigenvalue
$$
-\alpha_p = \max_{1\leq i \leq d} \lambda_i\left(\frac{M+M^T}{2}\right)
$$
with
$$
  M = \tilde{g}^{-\frac12}\ \Sigma^{-1}\,V^TU\,\tilde{g}^{\frac12},\quad \tilde{g} = U^T g U.
$$
To find $\alpha$ it suffices to maximize $\alpha_p$ over the domain $\mathcal{U}\subset M$.

Formulas for $\mu_{+}$ and $\mu_{-}$ in Theorems~\ref{theo: GEE non-expansive for positive rho} and \ref{theo: GEE non-expansive for negative rho} can be derived similarly. In the case of  $\nabla X$ non-singular, we obtain
\begin{equation*}
    \mu_+ =\lambda_{\max}\left(\frac{M_{+}+M_{+}^T}{2}\right),\qquad \mu_{-} =\lambda_{\max}\left(\frac{{M_{-}}+{M_{-}}^T}{2}\right),
    \end{equation*} 
    where $\lambda_{\max}(A)$ denotes the maximum eigenvalue $A$ and with
    \begin{equation*}
    M_{+}=-g^{1/2}(I-P_{X})\nabla X^{-1}g^{-1/2},\qquad M_{-}=-g^{1/2}P_{X}\nabla X^{-1}g^{-1/2}.
    \end{equation*}

\end{document}